\documentclass{article}
\usepackage{float}
\usepackage{graphicx}
\usepackage{pdfpages}
\usepackage{amsmath}
\usepackage{amsthm}
\usepackage{dsfont}
\usepackage{algpseudocode}
\usepackage{amssymb}
\usepackage{titlesec}
\usepackage{diagbox}
\usepackage{mathrsfs}
\usepackage{slashbox}
\usepackage{listings}
\usepackage{xcolor}
\usepackage{fancyhdr}
\usepackage{tikz}
\usepackage{marvosym}
\usepackage{url}
\usepackage[utf8]{inputenc}
\usepackage{wasysym}
\usepackage{hyperref}

\newtheorem{theorem}{Theorem}[section]   
\newtheorem{definition}[theorem]{Definition}  
\newtheorem{algorithm}[theorem]{Algorithm}  
\newtheorem{proposition}[theorem]{Proposition}  
\newtheorem{lemma}[theorem]{Lemma}  
\newtheorem{example}[theorem]{Example}  
\newtheorem{remark}[theorem]{Remark}  
\newtheorem{corollary}[theorem]{Corollary}  

\newtheorem{conjecture}{Conjecture}
\title{Polynomials in $\mathbf{Z}[x]$ whose divisors are enumerated by $SL_2(\mathbf{N}_0)$}
\author{Anton Shakov}
\date{February 2024}

\begin{document}
\maketitle

\begin{abstract}
We consider a certain left action by the monoid $SL_2(\mathbf{N}_0)$ on the set of divisor pairs $\mathcal{D}_f := \{ (m, n) \in \mathbf{N}_0 \times \mathbf{N}_0 : m \lvert f(n) \}$ where $f \in \mathbf{Z}[x]$ is a polynomial with integer coefficients. We classify all polynomials in $\mathbf{Z}[x]$ for which this action extends to an invertible map $\hat{F}_f: SL_2(\mathbf{N}_0) \rightarrow \mathcal{D}_f$ (Sections~\ref{section_class_statement} and~\ref{section_class_proof}). We call such polynomials \textit{enumerable}. One of these polynomials happens to be $f(n) = n^2 + 1$. It is a well-known conjecture that there exist infinitely many primes of the form $p = n^2 + 1$ (Conjecture~\ref{landau}). In Section~\ref{section_phi0recursions}, we construct a sequence $\mathcal{S}$ on the naturals defined by the recursions

$$
\begin{cases}
\mathcal{S}(4k) = 2\mathcal{S}(2k) - \mathcal{S}(k) \\
\mathcal{S}(4k+1) = 2\mathcal{S}(2k) + \mathcal{S}(2k+1) \\
\mathcal{S}(4k+2) =  2\mathcal{S}(2k+1) + \mathcal{S}(2k) \\
\mathcal{S}(4k+3) =  2\mathcal{S}(2k+1) - \mathcal{S}(k)  \\
\end{cases}
$$

\

with initial conditions $\mathcal{S}(1) = 0$, $\mathcal{S}(2) = 1$, $\mathcal{S}(3) = 1$. 

$$\{ \mathcal{S}(k) \}_{k \in \mathbf{N}} = \{0,1,1,2,3,3,2,3,7,8,5,5,8,7,3, \cdots \}$$

\

$\mathcal{S}$ is shown to have the  properties

\begin{enumerate}
    \item For all $n \in \mathbf{N}_0$, we have $\mathcal{S}(2^n) = \mathcal{S}(2^{n+1} - 1) = n$.
    \item For all $n \in \mathbf{N}_0$, the size of the fiber of $n$ under $\mathcal{S}$ satisfies $|\mathcal{S}^{-1}(\{n\})| = \tau(n^2 + 1)$ where $\tau$ is the divisor counting function.
    \item For all $n \in \mathbf{N}_0$, the integer $n^2 + 1$ is prime if and only if $\mathcal{S}^{-1}(\{n\}) = \{2^n, 2^{n+1} - 1\}$. 
    \item $\mathcal{S}(k)$ is a $2$-regular sequence (as introduced in \cite{q_regular_sequences}).
\end{enumerate}

Sequences with analogous properties are given for all enumerable polynomials in Section~\ref{section_summary}. Informally, these sequences encode the divisor structure of their respective polynomials. Property $(3)$ makes $\mathcal{S}$ amenable to the asymptotic analysis of $k$-regular sequences, an area which has undergone significant development in recent years.
\end{abstract}

\

\

\

\

\

\

\

\

\

\

\tableofcontents

\

\section{Introduction}\label{section_intro}

The set of nonnegative integer matrices with determinant 1, denoted $SL_2(\mathbf{N}_0)$, satisfies the axioms for a \textit{monoid}: 
it contains the usual identity matrix $I$ and is closed under the associative operation of matrix multiplication. We will see that for a small set of polynomials $f \in \mathbf{Z}[x]$, the monoid $SL_2(\mathbf{N}_0)$ intimately describes the divisors of $f$, namely the set of positive integers $d$ such that $d \lvert f(n)$ for some $n \in \mathbf{N}$.

\

An algebraic identity, first studied by Diophantus of Alexandria (c. 200 - c. 298) and which we will refer to as Diophantus' identity is given by

$$(a^2 + b^2)(c^2 + d^2) = (ac+bd)^2 + (ad-bc)^2$$

Diophantus' identity holds in any commutative ring, in particular, the integers $\mathbf{Z}$. It can be taken to mean that the set of sums of two squares of integers is closed under multiplication. Throughout history, it has found many fruitful generalizations. This includes Brahmagupta's Identity, Lagrange's identity, Euler's four-square identity, and Gauss's composition laws for binary quadratic forms.

\

In this paper, we consider Diophantus' identity in the context of the monoid $SL_2(\mathbf{N}_0)$. Taking this view, we can write

$$(a^2 + b^2)(c^2 + d^2) = (ac+bd)^2 + \det \begin{pmatrix} a & b \\
c & d \\
\end{pmatrix}^2$$

Assuming, 

$$\begin{pmatrix} 
a & b \\
c & d \\
\end{pmatrix} \in SL_2(\mathbf{N}_0)$$

yields the factorization 

\begin{equation}\label{connection}
 (ac + bd)^2 + 1 = (a^2 + b^2)(c^2 + d^2)   
\end{equation}

which is apparently a factorization identity for integers of the form $n^2+1$, with $n = ac + bd$.

\

There are still unanswered questions about the divisor structure of $n^2 + 1$ form numbers. Probably the best-known of these is

\begin{conjecture}\label{landau} (Landau's 4th Problem)
There are infinitely many primes of the form $p = n^2 + 1$
\end{conjecture}

Conjecture~\ref{landau} is a special case of what's known as the Bunyakovsky conjecture.

\begin{conjecture}\label{bunyakovsky} (Bunyakovsky Conjecture)
Let $f \in \mathbf{Z}[x]$ satisfy:

\begin{itemize}
    \item $f$ is irreducible in $\mathbf{Z}[x]$
    \item $f$ has a positive leading coefficient
    \item $\gcd(f(0), f(1), f(2), \cdots) = 1$
\end{itemize}

Then there exist infinitely many primes of the form $p = f(n)$.
\end{conjecture}

Identity~\ref{connection} motivates us to consider the correspondence between factorizations of integers of the form $n^2 + 1$ and the monoid $SL_2(\mathbf{N}_0)$.

A well-known property of $SL_2(\mathbf{N}_0)$ is that it is freely generated by the matrices

$$S := 
\begin{pmatrix}
    1 & 0 \\
    1 & 1 \\
\end{pmatrix}$$

$$T := 
\begin{pmatrix}
    1 & 1 \\
    0 & 1 \\
\end{pmatrix}$$

Hence, $SL_2(\mathbf{N}_0)$ is a free monoid of rank 2 (\cite{nathanson2014forest}, Corollary 1).

\begin{remark}
$S$ and $T$ are infinite order matrices in $SL_2(\mathbf{N}_0)$ and their powers are given by

$$S^\alpha = \begin{pmatrix}
    1 & 0 \\
    \alpha & 1 \\
\end{pmatrix} \text{ and  } \ T^\alpha = \begin{pmatrix}
    1 & \alpha \\
    0 & 1 \\
\end{pmatrix}$$

for all $\alpha \in \mathbf{N}_0$.
\end{remark}

This allows for a binary tree enumeration of $SL_2(\mathbf{N}_0)$ via multiplication by $S$ (left child) and $T$ (right child) beginning with the $2 \times 2$ identity matrix, $I$. The first 4 rows will be

$$
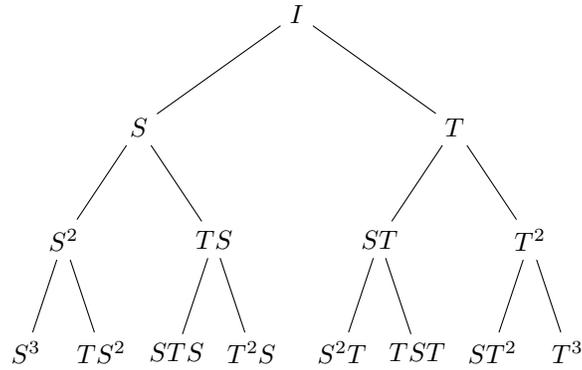
\begin{tikzpicture}[level distance=15mm,
level 1/.style={sibling distance=42mm},
level 2/.style={sibling distance=20mm},
level 3/.style={sibling distance=10mm}]
\node {$I$}
child {
node {$S$}
child {
node {$S^2$}
child {node {$S^3$}}
child {node {$TS^2$}}
}
child {
node {$TS$}
child {node {$STS$}}
child {node {$T^2S$}}
}
}
child {
node {$T$}
child {
node {$ST$}
child {node {$S^2T$}}
child {node {$TST$}}
}
child {
node {$T^2$}
child {node {$ST^2$}}
child {node {$T^3$}}
}
} ;
\end{tikzpicture}$$

We compute the explicit matrix products with the usual convention that matrix multiplication is done right to left.

$$
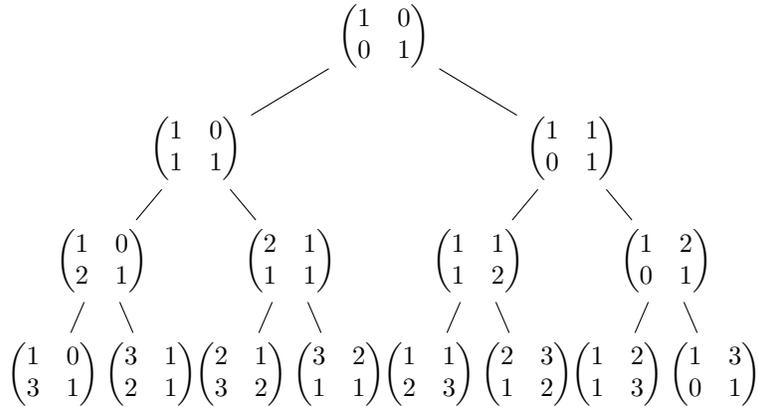
\begin{tikzpicture}[level distance=15mm, 
                    level 1/.style={sibling distance=50mm},
                    level 2/.style={sibling distance=25mm},
                    level 3/.style={sibling distance=13mm}]
  \node {\(\begin{pmatrix}1 & 0 \\ 0 & 1\end{pmatrix}\)}
    child {
      node {\(\begin{pmatrix}1 & 0 \\ 1 & 1\end{pmatrix}\)} 
      child {
        node {\(\begin{pmatrix}1 & 0 \\ 2 & 1\end{pmatrix}\)}
        child {node {\(\begin{pmatrix}1 & 0 \\ 3 & 1\end{pmatrix}\)}}
        child {node {\(\begin{pmatrix}3 & 1 \\ 2 & 1\end{pmatrix}\)}}
      }
      child {
        node {\(\begin{pmatrix}2 & 1 \\ 1 & 1\end{pmatrix}\)}
        child {node {\(\begin{pmatrix}2 & 1 \\ 3 & 2\end{pmatrix}\)}}
        child {node {\(\begin{pmatrix}3 & 2 \\ 1 & 1\end{pmatrix}\)}}
      }
    }
    child {
      node {\(\begin{pmatrix}1 & 1 \\ 0 & 1\end{pmatrix}\)}
      child {
        node {\(\begin{pmatrix}1 & 1 \\ 1 & 2\end{pmatrix}\)}
        child {node {\(\begin{pmatrix}1 & 1 \\ 2 & 3\end{pmatrix}\)}}
        child {node {\(\begin{pmatrix}2 & 3 \\ 1 & 2\end{pmatrix}\)}}
      }
      child {
        node {\(\begin{pmatrix}1 & 2 \\ 0 & 1\end{pmatrix}\)}
        child {node {\(\begin{pmatrix}1 & 2 \\ 1 & 3\end{pmatrix}\)}}
        child {node {\(\begin{pmatrix}1 & 3 \\ 0 & 1\end{pmatrix}\)}}
      } 
    };
\end{tikzpicture}$$

\begin{remark}\label{remark_interior}
The set of matrices which are on the "interior" of the tree, i.e. matrices with all nonzero components, is precisely the semigroup $SL_2(\mathbf{N})$.
\end{remark}

If we apply the mapping $SL_2(\mathbf{N}_0) \ni A \rightarrow (a^2 + b^2, ac + bd)$ to this matrix tree we get the tree of nonnegative integer pairs

$$
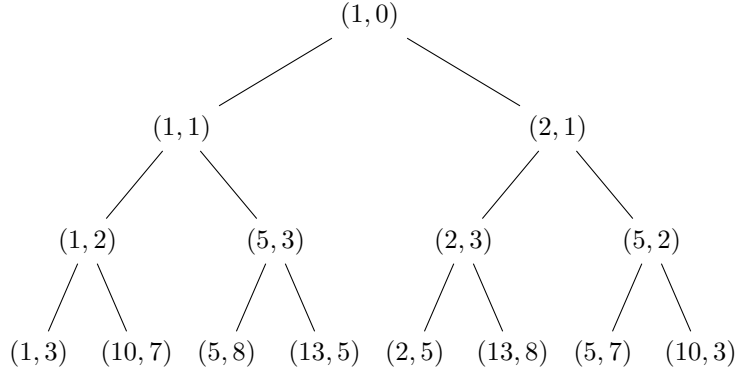
\begin{tikzpicture}[level distance=15mm, 
                    level 1/.style={sibling distance=50mm},
                    level 2/.style={sibling distance=25mm},
                    level 3/.style={sibling distance=13mm}]
  \node {\((1, 0)\)}
    child {
      node {\((1, 1)\)} 
      child {
        node {\((1, 2)\)}
        child {node {\((1, 3)\)}}
        child {node {\((10, 7)\)}}
      }
      child {
        node {\((5, 3)\)}
        child {node {\((5, 8)\)}}
        child {node {\((13, 5)\)}}
      }
    }
    child {
      node {\((2, 1)\)}
      child {
        node {\((2, 3)\)}
        child {node {\((2, 5)\)}}
        child {node {\((13, 8)\)}}
      }
      child {
        node {\((5, 2)\)}
        child {node {\((5, 7)\)}}
        child {node {\((10, 3)\)}}
      } 
    };
\end{tikzpicture}
$$

As a result of the identity $(ac + bd)^2 + 1 = (a^2 + b^2)(c^2 + d^2)$, we see that the first pair component $a^2 + b^2$ corresponds to a divisor of $n^2 + 1$ where $n = ac + bd$, or the second component of the pair $(a^2 + b^2, ac + bd)$. In Section~\ref{section_class_statement}, we clarify this connection and discuss which polynomials in $\mathbf{Z}[x]$, besides $n^2 + 1$, have this relationship to $SL_2(\mathbf{N}_0)$. 

\

We will prove in Section~\ref{section_class_proof} that the mapping $A \rightarrow (a^2 + b^2, ac + bd)$ is invertible with the codomain taken to be all pairs of nonnegative integers $(m, n)$ such that $m \lvert (n^2 + 1)$ (recall that $a \lvert b$ is the relation "$a$ divides $b$"). We will also prove that analogous mappings arising from other polynomials are invertible. 

\

In Section~\ref{section_examples}, we look at several clarifying examples and discuss a few properties of divisor pair trees.

\

In Section~\ref{section_phi0recursions} we specifically study the divisor pair tree given above and derive an integer sequence $\mathcal{S}$ which is shown to generate this tree.

\

In Section~\ref{section_summary}, we summarize our findings.

\section{Statement of classification theorem}\label{section_class_statement}

We state the following generalizations of Diophantus' identity. Let $\beta \in \mathbf{Z}$.

\[(a^2 + \beta ab + b^2)(c^2 + \beta cd + d^2) = (ac + \beta bc + bd)^2 + \beta (ac + \beta bc + bd)(ad-bc) + (ad - bc)^2\]
\[(a^2 + \beta ab - b^2)(c^2 + \beta cd - d^2) = (ac + \beta bc - bd)^2 + \beta (ac + \beta bc - bd)(ad-bc) - (ad - bc)^2\]

\

Just as in the case of Diophantus' identity, we can take this to mean that the sets $\{a^2 + \beta ab + b^2 : a, b \in \mathbf{Z} \}$ and $\{a^2 + \beta ab - b^2 : a, b \in \mathbf{Z} \}$ are closed under multiplication for all $\beta \in \mathbf{Z}$. We recover Diophantus' identity by setting $\beta = 0$ in the former identity. We again impose a restriction on the integers $a, b, c,$ and $d$ by assuming the matrix

$$\begin{pmatrix}
a & b \\
c & d \\
\end{pmatrix} \in SL_2(\mathbf{N}_0)$$ 

\

Since we know the determinant $ad - bc = 1$ we get

\[ (a^2 + \beta ab + b^2)(c^2 + \beta cd + d^2) = (ac + \beta bc + bd)^2 + \beta (ac + \beta bc + bd) + 1\]
\[(a^2 + \beta ab - b^2)(c^2 + \beta cd - d^2) = (ac + \beta bc - bd)^2 + \beta (ac + \beta bc - bd) - 1\]

\

These are now factorization identities for polynomials of the form

\begin{definition}\label{lilphiandpsi}
$\phi_\beta (n) := n^2 + \beta n + 1$, $\psi_\beta (n) := n^2 + \beta n - 1$. 
\end{definition} 

\begin{definition}
Let $f \in \mathbf{Z}[x]$. We define the divisor pair set of $f$ to be $\mathcal{D}_f := \{ (m, n) \in \mathbf{N}_0 \times \mathbf{N}_0 : m \lvert f(n) \}$. 
\end{definition}

\begin{definition}
Define the family of functions $\Phi = (\Phi_\beta)_{\beta \in \mathbf{N}_0}$ indexed by $\beta$ such that each $\Phi_\beta: SL_2(\mathbf{N}_0) \rightarrow \mathcal{D}_{\phi_\beta(n)}$ is given by

$$\Phi_\beta \begin{pmatrix}
    a & b \\
    c & d \\
\end{pmatrix} = (a^2 + \beta ab + b^2, ac + \beta bc + bd)$$
\end{definition}

Each $\Phi_{\beta}$ is a well-defined function since $(a^2 + \beta ab + b^2, ac + \beta bc + bd) \in \mathbf{N}_0 \times \mathbf{N}_0$ and we saw that

$$(a^2 + \beta ab + b^2) \ \lvert \ \phi_\beta(ac + \beta bc + bd)$$

\

Divisor pair sets allow us to investigate the matrix-factorization correspondence suggested by Diophantus' identity. For a given $\beta \in \mathbf{N}_0$, we ask: is $\Phi_\beta$ one-to-one? Is it onto? 

\

Note that the mapping we discussed at the end of Section~\ref{section_intro}, namely $SL_2(\mathbf{N}_0) \ni A \rightarrow (a^2 + b^2, ac + bd)$, is the function $\Phi_0$ with the codomain $\mathcal{D}_{\phi_0}$.

\

Before proceeding, we also define a family of functions for polynomials of the form $\psi_\beta(n)$.

\begin{definition}\label{Psifamily}
We define the family of functions $\Psi = (\Psi_\beta)_{\beta \in \mathbf{N}_0}$ indexed by $\beta$ such that each $\Psi_\beta : SL_2(\mathbf{N}_0) \rightarrow \mathcal{D}_{\psi_\beta(n)}$ is given by
$$\Psi_\beta \begin{pmatrix}
    a & b \\
    c & d \\
\end{pmatrix} = 
(\max \{a, b\}^2 + \beta ab - \min \{a, b \}^2,  \ \max \{ac, bd \} + \beta bc - \min \{ac, bd\})$$
\end{definition}

We show that the functions in this family are well-defined.

\begin{proof}
Clearly, both pair components are nonnegative at all matrices $A \in SL_2(\mathbf{N}_0)$. It remains to check that $\max \{a, b\}^2 + \beta ab - \min \{a, b \}^2$ divides $\psi_\beta(\max \{ac, bd \} + \beta bc - \min \{ac, bd\})$.

\

Since we're working in $SL_2(\mathbf{N}_0)$, the determinant $ad - bc = 1$ which means $a \geq b$ if and only if $c \geq d$, for all matrices except $I$ (this can easily be checked by assuming the contrary). Suppose $a = \max \{ a, b \}$. Then, $ac = \max \{ ac, bd \}$ and the conclusion follows from 
\[(a^2 + \beta ab - b^2)(c^2 + \beta cd - d^2) = (ac + \beta bc - bd)^2 + \beta (ac + \beta bc - bd) - 1\]
Otherwise, suppose $b = \max \{ a, b \}$. Then, $bd = \max \{ ac, bd \}$. The conclusion follows from the same identity with the variables renamed
\[(b^2 + \beta ab - a^2)(d^2 + \beta cd - c^2) = (bd + \beta bc - ac)^2 + 2(bd + \beta bc - ac)(ad - bc) - (ad-bc)^2\]

\end{proof}

\begin{definition}
Let $c: SL_2(\mathbf{N}_0) \rightarrow SL_2(\mathbf{N}_0)$ be an involution on $SL_2(\mathbf{N}_0)$ given by

\

$$c: \begin{pmatrix}
    a & b \\
    c & d \\
\end{pmatrix} \rightarrow 
\begin{pmatrix} 
d & c \\ 
b & a \\
\end{pmatrix}
$$

\

We call $c(A)$ the complement of $A \in SL_2(\mathbf{N}_0)$.
\end{definition}

\begin{remark}
$c(A)$ can also be written as an outer automorphism on $SL_2(\mathbf{N}_0)$

$$\begin{pmatrix}
    0 & 1 \\
    1 & 0 \\
\end{pmatrix}\begin{pmatrix}
    a & b \\
    c & d \\
\end{pmatrix}
\begin{pmatrix}
    0 & 1 \\
    1 & 0 \\
\end{pmatrix} = 
\begin{pmatrix} 
d & c \\ 
b & a \\
\end{pmatrix}
$$

where $\begin{pmatrix}
    0 & 1 \\
    1 & 0 \\
\end{pmatrix}$ has determinant $-1$.

\end{remark}

\

\begin{proposition}\label{refl}
Suppose $A = S^{\alpha_k} T^{\alpha_{k-1}} \cdots T^{\alpha_1}S^{\alpha_0} \in SL_2(\mathbf{N}_0)$ with each $\alpha_i \in \mathbf{N}_0$. Then $c(A) = T^{\alpha_k} S^{\alpha_{k-1}} \cdots S^{\alpha_1}T^{\alpha_0}$.
\end{proposition}

\begin{proof}
First, note that $c(S) = T$ and $c(T) = S$. One can then manually verify that for any matrices $A$ and $B$, we have $c(AB) = c(A)c(B)$.
\end{proof}

Note that Proposition~\ref{refl} can be interpreted as saying that $c$ is a  reflection across the $SL_2(\mathbf{N}_0)$ enumeration tree since the left child matrix, $S$ is swapped with the right child matrix, $T$.

\begin{proposition}\label{depmatrix}
For all $A \in SL_2(\mathbf{N}_0)$, we have $TA = c(S(c(A)))$ 
\end{proposition}

\begin{proof}

$$TA = c(c(TA)) = c(c(T) c(A)) = c(Sc(A))$$

One may also grasp this fact geometrically by viewing $c$ as a reflection across the matrix tree generated by $S$ and $T$.
\end{proof}

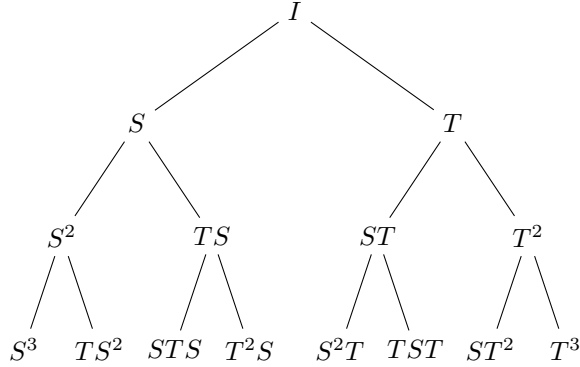
$$\begin{tikzpicture}[level distance=15mm,
level 1/.style={sibling distance=42mm},
level 2/.style={sibling distance=20mm},
level 3/.style={sibling distance=10mm}]
\node {$I$}
child {
node {$S$}
child {
node {$S^2$}
child {node {$S^3$}}
child {node {$TS^2$}}
}
child {
node {$TS$}
child {node {$STS$}}
child {node {$T^2S$}}
}
}
child {
node {$T$}
child {
node {$ST$}
child {node {$S^2T$}}
child {node {$TST$}}
}
child {
node {$T^2$}
child {node {$ST^2$}}
child {node {$T^3$}}
}
} ;
\end{tikzpicture}$$

\begin{proposition}\label{uniqueness of factors}
For all $A \in SL_2(\mathbf{N}_0)$, if $A \neq I$ then exactly one of the following conditions holds: 
\begin{itemize}
    \item $A = SB$ for some $B \in SL_2(\mathbf{N}_0)$
    \item $c(A) = SB$ for some $B \in SL_2(\mathbf{N}_0)$.
\end{itemize}
\end{proposition}

\begin{proof}
By Proposition~\ref{refl} exactly one of $A$ or $c(A)$ begins with $S$ (the other one has to begin with $T$).
\end{proof}

\begin{definition}\label{c_bar_f}
For a polynomial $f \in \mathbf{Z}[x]$, we define the involution $\bar{c}_f(m, n) := (\frac{|f(n)|}{m}, n)$. We call $\bar{c}_f$ the complement on $\mathcal{D}_f$ since it sends $m$ to its complementary factor. Note that $\bar{c}_f(m, n)$ is well-defined whenever $m \neq 0$.
\end{definition}

\begin{definition}
For a polynomial $f \in \mathbf{Z}[x]$, we define the left monoid action $(\cdot): SL_2(\mathbf{N}_0) \times \mathcal{D}_f \rightarrow \mathcal{D}_f$ satisfying

$$S  \ \cdot  \ (m, n) = (m, m+n)$$
$$T  \cdot (m, n) = \bar{c}_f (S  \cdot  \bar{c}_f(m, n))$$
\end{definition}

\begin{definition}
We say $F: SL_2(\mathbf{N}_0) \rightarrow \mathcal{D}_f$ is equivariant if for all $A \in SL_2(\mathbf{N}_0)$ we have

$$F(SA) = S \cdot F(A)$$
$$F(TA) = T \cdot F(A)$$

We will use the shorthand notations $\bar{S}(m, n) := S \cdot (m, n)$ and $\bar{T}(m, n) := T \cdot (m, n)$.
\end{definition}

Throughout this paper, we use bar $ \ \bar{} \ $ to distinguish between functions on matrices and functions on pairs. A bar indicates that we are talking about functions on pairs. We will also write $\bar{T}_f$ when we wish to specify the polynomial $f$ from which $\bar{c}_f$ and therefore $\bar{T}_f$ is derived.  

\begin{proposition}\label{invariance}
For all $f \in \mathbf{Z}[x]$ the set $\mathcal{D}_f$ is mapped to itself by the functions $\bar{S}$, $\bar{c}_f$, and $\bar{T}$ meaning that if $(m, n) \in \mathcal{D}_f$ then $\bar{S}(m, n) \in \mathcal{D}_f$, $\bar{c}_f(m, n) \in \mathcal{D}_f$, and $\bar{T}(m, n) \in \mathcal{D}_f$.
\end{proposition}

We omit the proof of Proposition~\ref{invariance} as it involves little more than writing out definitions.

\

To summarize, $\bar{S}$ doesn't depend on $f$, has infinite order, and preserves the first pair component $m$. The complement $\bar{c}_f$ has order 2, depends on $f$, and preserves the second pair component $n$. $\bar{T}$ depends on $f$, has infinite order, and doesn't preserve either pair component.

\begin{proposition}\label{prop_equivariant}
Let $F: SL_2(\mathbf{N}_0) \rightarrow \mathcal{D}_f$ be a mapping satisfying

$$F(SA) = \bar{S}\left(F(A)\right)$$
$$F\left(c(A)\right) = \bar{c}_f\left(F(A)\right)$$

for all $A \in SL_2(\mathbf{N}_0)$. Then $F$ is equivariant.
\end{proposition}

\begin{proof}
Since $\bar{T}(m, n) = \bar{c}_f\left(\bar{S}\left(\bar{c}_f(m, n)\right)\right)$.
\end{proof}

It is easier to use Proposition~\ref{prop_equivariant} to check that a map $F$ is equivariant than to proceed straight from the definition.

\begin{proposition}
Each function in the families $\Phi = (\Phi_\beta)_{\beta \in \mathbf{N}_0}$ and $\Psi = (\Psi_\beta)_{\beta \in \mathbf{N}_0}$ is equivariant.
\end{proposition}

\begin{proof}
We prove the result for the family $\Phi$. The proof for the family $\Psi$ is largely identical.

\

For,
$$A = \begin{pmatrix}
a & b \\
c & d \\
\end{pmatrix} \in SL_2(\mathbf{N}_0)$$

we check that, $\Phi_\beta(SA) = \bar{S}(\Phi_\beta(A))$ and $\Phi_\beta\left(c(A)\right) = \bar{c}_f(\Phi_\beta(A))$ and then use Proposition~\ref{prop_equivariant} to conclude that $\Phi_\beta$ is equivariant.

\begin{align*}
\Phi_\beta(SA) &=  \Phi_\beta \begin{bmatrix} \begin{pmatrix}
1 & 0 \\
1 & 1 \\
\end{pmatrix} 
\begin{pmatrix}
a & b \\
c & d \\
\end{pmatrix} 
\end{bmatrix} \\
               &=\Phi_\beta \begin{pmatrix}
                            a & b \\
                            a+c & b+d \\
                            \end{pmatrix} \\
               &= (a^2 + \beta ab + b^2, \ a(a+c) + \beta b(a+c) + b(b+d)\\
               &= (a^2 + \beta ab + b^2, \ (ac + \beta bc + bd) + (a^2 + \beta ab + b^2)) \\
               &= \bar{S}(\Phi_\beta(A))
\end{align*}

\begin{align*}
\Phi_\beta(c(A)) &= \Phi_\beta\begin{pmatrix}
    d & c \\
    b & a \\
\end{pmatrix} \\
        &= (d^2 + \beta dc + c^2, \ db + \beta cb + ca) \\
        &= (\frac{\phi_\beta(ac + \beta bc + bd)}{a^2 + \beta ab + b^2}, ac + \beta bc + bd) \\
        &=\bar{c}_f(\Phi_\beta(A))
\end{align*}
\end{proof}

We now define enumerable polynomials.

\begin{definition}
For a given $f \in \mathbf{Z}[x]$, we say $\mathcal{D}_f$ can be enumerated by $SL_2(\mathbf{N}_0)$ if there exists an invertible, equivariant map $F: SL_2(\mathbf{N}_0) \rightarrow \mathcal{D}_f$. If $\mathcal{D}_f$ can be enumerated by $SL_2(\mathbf{N}_0)$, we say the polynomial $f$ is enumerable. 
\end{definition}

For Lemmas~\ref{lemma_(0,n)} and~\ref{lemma_F(I)=(1,0)} as well as Corollary~\ref{cor_nonvanishing}, we assume that $F: SL_2(\mathbf{N}_0) \rightarrow \mathcal{D}_f$ is an invertible, equivariant map for some enumerable polynomial $f \in \mathbf{Z}[x]$.

\begin{lemma}\label{lemma_(0,n)}
The pair $(0,n) \notin \mathcal{D}_f$
 for all  $n \in \mathbf{N}_0$.
\end{lemma}

\begin{proof}
Since $F$ is assumed to be invertible, it must be onto and so $\mathbf{Im}(F) = \mathcal{D}_f$ where $\mathbf{Im}(F)$ denotes the image of $F$. Suppose towards a contradicition that $(0, n) \in \mathbf{Im}(F)$ i.e. that there exists a matrix $A \in SL_2(\mathbf{N}_0)$ such that $F(A) = (0, n)$. Then $F(SA) = \bar{S}(F(A)) = (0, n + 0) = F(A)$ contradicting the one-to-one assumption on $F$.
\end{proof}

\begin{corollary}\label{cor_nonvanishing}
The polynomial $f$ is nonvanishing on $\mathbf{N}_0$, i.e. $f(n) \neq 0$ for all $n \in \mathbf{N}_0$.
\end{corollary}

\begin{proof}
Otherwise, there is an integer $n$ such that $0 = f(n)$ and so $0 \lvert f(n)$ which means $(0, n) \in \mathcal{D}_f$, contradicting the previous lemma.
\end{proof}

\begin{lemma}\label{lemma_F(I)=(1,0)}
$F$ sends the identity matrix $I$ to the pair $(1,0)$.
\end{lemma}

\begin{proof}
The pair $(|f(0)|, 0) \in \mathcal{D}_f = \mathbf{Im}(F)$. Then there exists a matrix $A \in SL_2(\mathbf{N}_0)$ such that $F(A) = (|f(0)|, 0)$. We will show that $A = I$ and $|f(0)| = 1$. 

\

Suppose $A = SB$ for some $B \in SL_2(\mathbf{N}_0)$. Then, $F(A) = F(SB) = \bar{S}(F(B)) = (|f(0)|, 0)$. We rewrite $F(B) := (m, n) \in \mathcal{D}_f$. Then $\bar{S}(F(B)) = (m, n+m) = (|f(0)|, 0) \implies n < 0$ (since $f$ being nonvanishing implies $|f(0)|  
> 0$) but this is a contradiction since $n \in \mathbf{N}_0$.

\

Suppose instead that $c(A) = SB$ for some $B \in SL_2(\mathbf{N}_0)$. Then we have $F(c(A)) = \bar{c}_f(F(A)) = (\frac{|f(0)|}{|f(0)|}, 0) = (1,0)$. So, $F(SB) = (1,0)$. Write $F(B) := (m, n) \in \mathcal{D}_f$. Then, $F(SB) =  \bar{S}(F(B)) = (m, m+n) = (1,0) \implies n = -1$, a contradiction.

\

It follows from Proposition~\ref{uniqueness of factors}  that $A = I$ and therefore that $F(I) = (|f(0)|, 0)$. We also have $F\left(c(I)\right) = \bar{c}_f\left(F(I)\right)$, and since $c(I) = I$, we get $F(I) = \bar{c}_f\left(F(I)\right)$. Finally we can write $F(I) = (|f(0)|, 0) = \bar{c}_f(|f(0)|, 0)) = (\frac{|f(0)|}{|f(0)|}, 0) = (1, 0)$.
\end{proof}

We can now prove the uniqueness of invertible, equivariant maps.

\begin{proposition}\label{uniqueness of st pres map}
Suppose $F$ and $F'$ are both invertible, equivariant maps into the set $\mathcal{D}_f$ derived from the same polynomial $f \in \mathbf{Z}[x]$. Then $F = F'$.
\end{proposition}

\begin{proof}
Since the matrices $S$ and $T$ freely generate the monoid $SL_2(\mathbf{N}_0)$, we see that an equivariant map $F$ is uniquely determined by where it sends the identity matrix. However, we proved that an invertible, equivariant map $F$ necessarily sends the identity matrix $I$ to the pair $(1,0)$. The conclusion follows.
\end{proof}

The uniqueness of invertible, equivariant maps motivates the following definition

\begin{definition}
For a polynomial $f \in \mathbf{Z}[x]$ we let $\hat{F}_f: SL_2(\mathbf{N}_0) \rightarrow \mathcal{D}_f$ denote the equivariant map satisfying $\hat{F}_f(I) = (1,0)$.
\end{definition}

Note that we do not require $\hat{F}_f$ to be invertible. Whenever a polynomial $f$ is enumerable, $\hat{F}_f$ is automatically forced to be invertible by the uniqueness of equivariant, invertible maps. However, it is not hard to see that $\hat{F}_f: SL_2(\mathbf{N}_0) \rightarrow \mathcal{D}_f$ will more generally be well-defined for polynomials $f \in \mathbf{Z}[x]$ that are nonvanishing on $\mathbf{N}_0$.

\

We can write down the first few rows of the general binary tree generated by $\hat{F}_f$ through applying $\bar{S}$ as the left child and $\bar{T}_f$ as the right child

$$
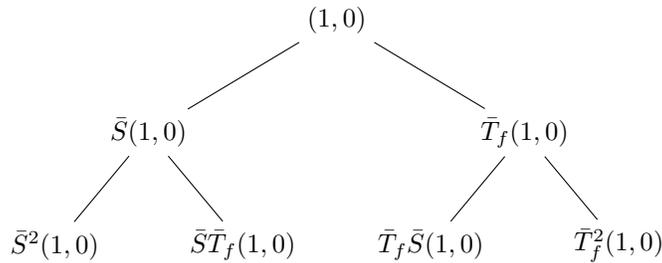
\begin{tikzpicture}[level distance=15mm, 
                    level 1/.style={sibling distance=50mm},
                    level 2/.style={sibling distance=25mm}]
  \node {\((1, 0)\)}
    child {
      node {\(\bar{S}(1,0)\)} 
      child {
        node {\(\bar{S}^2(1,0)\)}
      }
      child {
        node {\(\bar{S}\bar{T}_f(1,0)\)}
      }
    }
    child {
      node {\(\bar{T}_f(1,0)\)}
      child {
        node {\(\bar{T}_f\bar{S}(1,0)\)}
      }
      child {
        node {\(\bar{T}_f^2(1,0)\)}
      } 
    };
\end{tikzpicture}
$$

\

We can evaluate each composition and write each pair in terms of $f$.

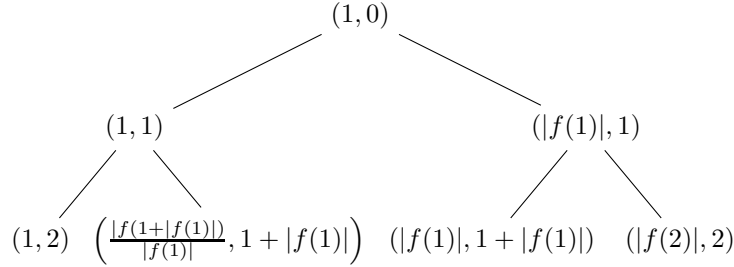
$$
\begin{tikzpicture}[level distance=15mm, 
                    level 1/.style={sibling distance=60mm},
                    level 2/.style={sibling distance=25mm}]
  \node {\((1, 0)\)}
    child {
      node {\((1, 1)\)} 
      child {
        node {\((1, 2)\)}
      }
      child {
        node {\(\left(\frac{|f(1 + |f(1)|)}{|f(1)|}, 1 + |f(1)|\right)\)}
      }
    }
    child {
      node {\(\left(|f(1)|, 1\right)\)}
      child {
        node {\(\left(|f(1)|, 1 + |f(1)|\right)\)}
      }
      child {
        node {\(\left(|f(2)|, 2\right)\)}
      } 
    };
\end{tikzpicture}
$$

Notice that the "boundary" of the binary tree enumerating $\hat{F}_f$ will be $(1, n)$ on the left and $(|f(n)|, n)$ on the right. This corresponds to the \textit{trivial} factorizations of $\mathcal{D}_f$, namely those of the form $|f(n)| = 1 \times |f(n)|$. The interior of the tree (i.e. all the pairs not in the boundary) correspond to the \textit{nontrivial} factorizations of the tree (i.e. where neither $m$ nor $\frac{|f(n)|}{m}$ equals $1$). We have previously remarked that the interior of the $S, T$ tree enumeration for $SL_2(\mathbf{N}_0)$ will be the semigroup $SL_2(\mathbf{N})$. Thus, nontrivial factorizations of $f(n)$ for $f$ enumerable and $n \in \mathbf{N}_0$ will correspond to matrices in $SL_2(\mathbf{N})$.

\begin{lemma}\label{f(1), f(2) prime}
Suppose that $f \in \mathbf{Z}[x]$ is enumerable. Then $|f(1)|$, $|f(2)|$, and $|f(|f(1)|)|$ are prime.
\end{lemma}

\begin{proof}
We show that $|f(1)|$, $|f(2)|$, and $|f(|f(1)|)|$ do not appear in the interior of the tree for $\hat{F}_f$. From the definition of $\bar{c}_f$ we see that the $\hat{F}_f$ generalized divisor pair tree will be symmetric in the second component. From the definition of $\bar{S}$ and $f$ being nonvanishing, we have that children pairs must have a strictly larger second component than their parent. As a result, $1 + |f(1)|$ is the smallest second component on the interior of the tree, and since $1 < 1 + |f(1)|$, we have that $|f(1)|$ only has a nontrivial factorization on the tree. Since $f$ is enumerable, which of course means that $\hat{F}_f$ is invertible, we have that $|f(1)|$ has no nontrivial factorizations in general, and so it must be prime. Similarly, since $|f(1)| < 1 + |f(1)|$, we find that $|f(|f(1)|)|$ is prime. Since $\hat{F}_f$ is one-to-one, $ 1 < |f(1)|$ (otherwise $(1,1)$ appears twice on the second row). Then, $2 < 1 + |f(1)|$. Thus, $|f(2)|$ must also be prime.
\end{proof}

\begin{lemma}\label{f -f lemma}
Let $f \in \mathbf{Z}[x]$ be nonvanishing on $\mathbf{N}_0$ so that $\hat{F}_f$ is well-defined. Then $\hat{F}_f(A) = \hat{F}_{(-f)}(A)$ for all matrices $A \in SL_2(\mathbf{N}_0)$.
\end{lemma}

\begin{proof}
$\hat{F}_f$ is determined by $\bar{S}$ (which is independent of $f$) and $\bar{c}_f(m,n) = (\frac{|f(n)|}{m}, n) = (\frac{|-f(n)|}{m}, n) = \bar{c}_{(-f)}(m,n)$. Since $\hat{F}_f(I) = (1,0)$, the two functions must be the same at every matrix.
\end{proof}

Lemma~\ref{f -f lemma} allows us to consider $f \in \mathbf{Z}[x]$ up to multiples of $\{\pm 1\}$. Note that one of $f$ or $-f$ must have a positive leading coefficient and so we may restrict our focus to polynomials $f \in \mathbf{Z}[x]$ with positive leading coefficients throughout the classification.

\

We can finally state the classification result which will be proved in the following section. Recall that we defined $\phi_\beta(n) = n^2 + \beta n + 1$ and $\psi_\beta(n) = n^2 + \beta n - 1$.

\begin{theorem}\label{main}
Suppose that $f \in \mathbf{Z}[x]$ is enumerable and, without loss of generality, that $f$ has a positive leading coefficient. Then

$$f \in \{\phi_0, \ \phi_1, \ \psi_2, \ \phi_3 \}$$

with $\hat{F}_{\phi_0} = \Phi_0, \ \hat{F}_{\phi_1} = \Phi_1, \ \hat{F}_{\psi_2} = \Psi_2$,  and  $\hat{F}_{\phi_3} = \Phi_3$.
\end{theorem}

\section{Proof of classification theorem}\label{section_class_proof}

\begin{remark}
The inverses of $\bar{S}$ and $\bar{c}_f$ are given by, $\bar{S}^{-1}(m,n) = (m, n-m)$, and $\bar{c}_f^{-1}(m, n) = \bar{c}_f(m, n)$.
\end{remark}

\begin{proposition}\label{maincondition}
Let $f \in \mathbf{Z}[x]$ be nonvanishing on $\mathbf{N}_0$. Then $f$ is enumerable if and only if for all $(m, n) \in \mathcal{D}_f \setminus \{(1, 0) \}$

$$\min \{m, \frac{|f(n)|}{m} \} \leq n < \max \{m, \frac{|f(n)|}{m} \}$$

\end{proposition}

\begin{proof}
($\implies$:) \ Suppose $f$ is enumerable. Then there exists an invertible, equivariant map $F: SL_2(\mathbf{N}_0) \rightarrow \mathcal{D}_f$. Consider $ (m, n) \in \mathcal{D}_f \setminus \{(1,0)\}$. Since $F$ is onto, there exists $A \in SL_2(\mathbf{N}_0)$ such that $(m, n) = F(A)$. By Proposition~\ref{uniqueness of factors} and $F$ being invertible \textit{exactly one} of the following conditions hold

\begin{enumerate}
    \item $F(A) = F(SB) = \bar{S}(F(B))$ for some $B \in SL_2(\mathbf{N}_0)$
    \item $F(c(A)) = F(SB) = \bar{S}(F(B))$ for some $B \in SL_2(\mathbf{N}_0)$
\end{enumerate}

\

Note that

\

$(1) \iff F(B) = \bar{S}^{-1}(F(A)) = \bar{S}^{-1}(m, n) = (m, n-m) \in \mathcal{D}_f$.

\

$(2) \iff F(B) = \bar{S}^{-1}(\bar{c}_f(F(A))) = \bar{S}^{-1}(\frac{|f(n)|}{m}, n) = (\frac{|f(n)|}{m}, n - \frac{|f(n)|}{m}) \in \mathcal{D}_f$.

\

Of course, if $m \lvert f(n)$ then

$$m \lvert f(n-m)$$

and,

$$\frac{|f(n)|}{m} \ \lvert  \  f(n - \frac{|f(n)|}{m} )$$

meaning the only way $(1)$ or $(2)$ could fail is if $n - m < 0$ or $n - \frac{|f(n)|}{m} < 0$, respectively (as that would mean leaving $\mathbf{N}_0 \times \mathbf{N}_0$). Thus, exactly one of these two inequalities holds. This is equivalent to the conclusion of this direction.

\

$(\impliedby:)$ We will assume that for all $(m, n) \in \mathcal{D}_f \setminus \{(1, 0) \}$ 

$$\min \{m, \frac{|f(n)|}{m} \} \leq n < \max \{m, \frac{|f(n)|}{m} \}$$

and prove that $\hat{F}_f: SL_2(\mathbf{N}_0) \rightarrow \mathcal{D}_f$ is invertible. Since $f$ is nonvanishing on $\mathbf{N}_0$, $\hat{F}_f$ is well-defined. By definition, $\hat{F}_f(I) = (1,0)$. By the previous direction, we know our assumption is equivalent to saying that for all $(m, n) \in \mathcal{D}_f \setminus \{(1,0)\},$ exactly one of the following holds: $n - m \geq 0$ or $n - \frac{|f(n)|}{m} \geq 0$. This means that exactly one of the following holds: $\bar{S}^{-1}(m, n) \in \mathcal{D}_f$ or $\bar{S}^{-1}(\bar{c}_f(m,n)) \in \mathcal{D}_f$.

\

If we start with a given $(m, n) \in \mathcal{D}_f$ and iterate the process of "factoring out" an $\bar{S}^{-1}$ or $\bar{c}_f$, depending on which one keeps us in $\mathcal{D}_f$, we are guaranteed by our assumption to get an expression of the form

$$\bar{S}^{-\alpha_k} \bar{c}_f \bar{S}^{-\alpha_{k-1}} \bar{c}_f \cdots \bar{c}_f \bar{S}^{-\alpha_1} \bar{c}_f \bar{S}^{-\alpha_0}(m,n) = (1,0)$$

where each $\alpha_i \geq 0$. Then simply "unwrap" the expression by undoing the inverses to find $\hat{F}_f^{-1}(m, n) \in SL_2(\mathbf{N}_0)$.

\end{proof}

\

In fact, one can infer from this argument that

\begin{remark}\label{onto, one-to-one}

\begin{itemize}

\

    \item $\hat{F}_f$ is one-to-one $\iff n < \max \{m, \frac{|f(n)|}{m} \}$ for all $(m, n) \in \mathbf{Im}(\hat{F}_f)$ \\
    \item $\hat{F}_f$ is onto $\iff \min \{m, \frac{|f(n)|}{m} \} \leq n$ for all $(m, n) \in \mathcal{D}_f \setminus \{(1,0)\}$
\end{itemize}
\end{remark}

\

Note that the first inequality in Propositon~\ref{maincondition} is $(=)$ precisely when $(m, n) = (1, 1)$ since we showed $|f(0)| = 1$ and so $\gcd(n, f(n)) = 1$ for all $n \in \mathbf{N}_0$. Hence, if $m \lvert f(n)$ and $m = n$ we must have $(m, n) = (1,1)$.

\

We will use the condition in Proposition~\ref{maincondition} to prove Theorem~\ref{main}.  

\

\begin{lemma}
Suppose that $f \in \mathbf{Z}[x]$ is enumerable. Then $\deg(f) \geq 2$.
\end{lemma}

\begin{proof}
Suppose, towards a contradiction, that there exists an invertible, equivariant map $F: SL_2(\mathbf{N}_0) \rightarrow \mathcal{D}_f$ for some $f$ with $\deg(f) \leq 1$.

\

We first address the case $\deg(f) = 0$. We know $|f(0)| = 1$ so the only two possibilities for enumerable polynomials are the constant functions $\pm 1$. In either case, we have $F(S) = F(SI) =  \bar{S}(1,0) = (1,1)$ and $F(T) = F(TI) = \bar{T}_f(1,0) = (\bar{c}_f \circ \bar{S} \circ \bar{c}_f) (1,0) = (1,1)$, contradicting that $F$ is one-to-one.

\

We now address the case $\deg(f) = 1$. Since $|f(0)| = 1$, we can write $f(n) = an \pm 1$ for some $a \in \mathbf{N}$ (we're allowed to make the restriction to $\mathbf{N}$ by Lemma~\ref{f -f lemma}). We find $(m, n) \in \mathcal{D}_f$ such that the inequalities in Proposition~\ref{maincondition} fail.

\

For $f(n) = an+1$, take $n = a+2$. Then $f(a+2) = (a+1)^2 \implies (m, n) := (a+1, a+2) \in \mathcal{D}_f$. Then, $\max \{m, \frac{|f(n)|}{m} \} = a+1 < a+2$, contradicting the right inequality in Proposition~\ref{maincondition}.

\

For $f(n) = an-1$, take $n = a^3$. Then $f(a^3) = a^4 - 1 = (a^2 - 1)(a^2 + 1) \implies (a^2 + 1, a^3) \in \mathcal{D}_f$. $\max \{a^2-1, a^2+1 \} = a^2 + 1 < a^3$ for all $a \geq 2$. Hence, it remains to check $a = 1$ (i.e. $f(n) = n-1$). But then $f(1) = 0$ which contradicts $f$ being nonvanishing on $\mathbf{N}_0$.
\end{proof}

\begin{lemma}
Suppose $f$ is monic and $\deg(f) = 2$. Then $f$ is enumerable if and only if $f \in \{\phi_0, \phi_1, \psi_2, \phi_3\}$
\end{lemma}

\begin{proof}

($\implies:$)
Write $f(n) = n^2 \pm bn \pm 1$ with $b > 0$. Then we have the following forms to consider

\begin{enumerate}
    \item $f(n) = n^2 + bn + 1$
    \item $f(n) = n^2 + bn - 1$
    \item $f(n) = n^2 - bn + 1$
    \item $f(n) = n^2 - bn - 1$
\end{enumerate}

$(3)$ and $(4)$ can be eliminated by remarking that for both forms $|f(b)| = 1$ and so $(1, b) = (|f(b)|, b) 
= \bar{c}_f(1, b)$ which contradicts $F$ being one-to-one.

\

For form $(1)$, if $b$ is even, we have that $f(1) = b+2$ is also even, and since $|f(1)|$ must be prime (Lemma~\ref{f(1), f(2) prime}), we get $|b+2| = 2$ and so $b = 0$ is the only admissible even value of $b$. For the case where $b$ is odd, we can check that

$$f(n_0) = n_0^2 + bn_0 + 1 = (n_0 + \frac{b-1}{2})^2$$

where $n_0 := (\frac{b-1}{2})^2 - 1$. If $b > 3$ and odd we have that $n_0 > 0$ and so 

$$\bar{c}_f(n_0 + \frac{b-1}{2}, \ n_0) =  (n_0 + \frac{b-1}{2}, n_0)$$

implies $f(n) = n^2 + bn + 1$ is not enumerable for $b > 3$ and odd. Thus for form $(1)$ we have narrowed the possibilities down to $\phi_0$, $\phi_1$ and $\phi_3$.

\

For form $(2)$, set $n = b-1$. Then, $f(b-1) = b(2b-3)$ meaning the first inequality of proposition~\ref{maincondition} fails for all $b > 2$ since we have

$$\min \{b, 2b-3 \} > b-1$$

\

Thus, it remains to check $b=1$ for form $(2)$, or whether $f(n) = n^2 + n - 1$ is enumerable. However, $f(1) = 1$, which contradicts $\hat{F}_f$ being one-to-one. Thus, we have narrowed the possibilities down to $\psi_2$.

\

($\impliedby:$) \ We now apply Proposition~\ref{maincondition} to prove that each of the polynomials $f \in \{\phi_0, \phi_1, \psi_2, \phi_3\}$ is enumerable by showing that all pairs $(m, n) \in \mathcal{D}_f \setminus \{(1,0)\}$ satisfy the bounds

\

$$\min \{m, \frac{|f(n)|}{m} \} \leq n < \max \{m, \frac{|f(n)|}{m} \}$$

\

($\phi_0:$) \ Consider $(m, n) \in \mathcal{D}_{\phi_0} \setminus \{(1,0)\}$. If the right-side inequality fails, i.e. $n \geq m$ and $n \geq \frac{|\phi_0(n)|}{m}$

\

$$|\phi_0(n)| = |n^2 + 1| = m \frac{|\phi_0(n)|}{m} \leq n \cdot n = n^2$$

Which is false for all $n \in \mathbf{N}_0$.

\

If the left-side inequality fails, i.e. $m \geq n+1$ and $\frac{|\phi_0(n)|}{m} \geq n+1$

\

$$|\phi_0(n)| = |n^2 + 1| = m \frac{|\phi_0(n)|}{m} \geq (n+1)^2$$

which is always false on $\mathbf{N}_0$, except at $n = 0$ which corresponds to the pair $(1,0)$. Thus, by Proposition~\ref{maincondition}, $\phi_0$ is enumerable.

\

\

($\phi_1:$) Consider $(m, n) \in \mathcal{D}_{\phi_1} \setminus \{(1,0)\}$.
If for some pair $(m, n) \in \mathcal{D}_f$ the right-side inequality fails, we get the same contradiction as for $\phi_0$.

\

If the right-side inequality fails i.e. $m \geq n+1$ and $\frac{|\phi_1(n)|}{m} \geq n+1$

$$|\phi_1(n)| = |n^2 + n + 1| = m\frac{|\phi_1(n)|}{m} \geq (n+1)^2$$

which is always false on $\mathbf{N}_0$, except at $n = 0$ which corresponds to the pair $(1,0)$. Thus, by Proposition~\ref{maincondition}, $\phi_1$ is enumerable.

\

\

($\psi_2:$) Consider $(m, n) \in \mathcal{D}_{\psi_2} \setminus \{(1,0)\}$. If the right-side inequality fails

$$|\psi_2(n)| = |n^2 + 2n - 1| = m \frac{|\psi_2(n)|}{m} \leq n \cdot n = n^2$$

which is false for all $n \in \mathbf{N}_0$ except $n=0$, which corresponds to the pair $(1,0)$.

\

If the left-side inequality fails, i.e. $m \geq n+1$ and $\frac{|\psi_2(n)|}{m} \geq n+1$

$$|\psi_2(n)| = |n^2 + 2n - 1| = m\frac{|\psi_2(n)|}{m} \geq (n+1)^2$$

which is false for all $n \in \mathbf{N}_0$. Thus, by Proposition~\ref{maincondition}, $\psi_2$ is enumerable.

\

\

($\phi_3:$) Consider $(m, n) \in \mathcal{D}_{\phi_3} \setminus \{(1,0)\}$. If the right-side inequality fails, we get the same contradiction as for $\phi_0$ and $\phi_1$.

\

If the left-side inequality fails i.e. $m \geq n + 1$ and $\frac{|\phi_3(n)|}{m} \geq n+1$, we have $n^2 + 3n + 1 \geq (n+1)^2$. This will be true in general on $\mathbf{N}_0$. However, it will also be true that

$$n^2 + 3n + 2 = (n+1)(n+2) > |n^2 + 3n + 1| \geq (n+1)^2$$

Thus, we get

$$(n+1)(n+2) > m  \frac{|\phi_3(n)|}{m} \geq (n+1)^2$$

and since $m$ and $\frac{f(n)}{m}$ are necessarily integers, both $\geq n+1$, we have that $m = \frac{|\phi_3(n)|}{m} = n+1 \implies (n+1)^2 = |n^2 + 3n + 1| \implies n = 0$. Once again, this corresponds to the pair $(1,0)$. Thus, by Proposition~\ref{maincondition}, $\phi_3$ is enumerable.
\
\end{proof}

\begin{lemma}\label{carlslemma}
Suppose $f \in \mathbf{Z}[x]$ is enumerable . Then $\deg(f) \leq 2$ and if $\deg(f) = 2$ then $f$ has leading coefficient $\pm 1$.
\end{lemma}

\begin{proof}
First, we'll use Lemma~\ref{f -f lemma} to narrow our attention to $f$ with positive leading coefficients.

\

We will show that if $f \in \mathbf{Z}[x]$ with either $\deg(f) \geq 3$ or $\deg(f) = 2$ and leading coefficient at least $2$, then there exists $n_0 \in \mathbf{N}$ such that $f(n_0) = (n_0+a)(n_0+b)$ for $a, b \in \mathbf{N}$. From this, it follows that no such $f$ can be enumerable, since the left-side inequality is violated

$$\min \{n+a, n+b \} > n$$

\

The following argument is due to Carl Schildkraut \cite{carlslemma}

\

Pick a positive integer $a$, large enough so that

\begin{enumerate}
    \item $|f(-a)| > 3a$, and
    \item $f(n) > \frac{3}{2}n^2$ for all $n > 2a$.
\end{enumerate}

We can choose such an integer $a$ since $f$ grows at least on the order of $2|n|^2$ as $|n| \rightarrow \infty$.

\

We choose $n_0 := |f(-a)| - a$. Note that $n_0 > 2a$ by $(1)$. In particular, $n_0 > 0$. We find

$$f(n_0) = f(|f(-a)|-a) \equiv f(-a) \equiv 0 \text{ (mod } |f(-a)|)$$

As a result, we can write

$$f(n_0) = |f(-a)|(n_0+b) = (n_0+a)(n_0+b)$$

for some $b \in \mathbf{Z}$. We want to show $b > 0$. Indeed, we have

$$f(n_0) > \frac{3}{2}n_0^2 > n_0(n_0 + a)$$

by $(2)$ and then $(1)$. Therefore $b > 0$ and the conclusion follows.
\end{proof}

\section{Some examples}\label{section_examples}

\begin{theorem}\label{divisorcount}

Let $\tau(n)$ denote the divisor counting function. Then for $n \in \mathbf{N}_0$
\

\begin{flalign*}
&& \# \left\{ \begin{pmatrix} a & b \\
    c & d \\
    \end{pmatrix} \in SL_2(\mathbf{N}_0) : ac + bd = n \right\} &= \tau(n^2 + 1) &
\end{flalign*}
    
\begin{flalign*}
&& \# \left\{ \begin{pmatrix} a & b \\
    c & d \\
    \end{pmatrix} \in SL_2(\mathbf{N}_0) : ac + bc + bd = n \right\} &= \tau(n^2 + n + 1) &
\end{flalign*}
    
\begin{flalign*}
&& \# \left\{ \begin{pmatrix} a & b \\
    c & d \\
    \end{pmatrix} \in SL_2(\mathbf{N}_0) : ac + 3bc + bd = n \right\} &= \tau(n^2 + 3n + 1) &
\end{flalign*}

\begin{flalign*}
&& \# \left\{ \begin{pmatrix} a & b \\
    c & d \\
    \end{pmatrix} \in SL_2(\mathbf{N}_0) : \max \{ac, bd \} + 2bc - \min \{ac, bd\} = n \right\} &= \tau(n^2 + 2n - 1) &
\end{flalign*}

\end{theorem}

\begin{proof}
This follows immediately from our proof that every $f \in \{\phi_0, \phi_1, \psi_2, \phi_3 \}$ is enumerated by its corresponding function in the families $\Phi$ or $\Psi$.
\end{proof}

Restricting to the interior of $SL_2(\mathbf{N}_0)$, namely to the set $SL_2(\mathbf{N})$ (which corresponds to nontrivial factorization pairs) we show

\begin{corollary}
    \item $\phi_0(n) = n^2 + 1$ is prime $\iff \nexists \begin{pmatrix} a & b \\
    c & d \\
    \end{pmatrix} \in SL_2(\mathbf{N})$ such that $ac + bd = n$ \\
    \item $\phi_1(n) = n^2 + n + 1$ is prime $\iff \nexists \begin{pmatrix} a & b \\
    c & d \\
    \end{pmatrix} \in SL_2(\mathbf{N})$ such that $ac + bc + bd = n$ \\
    \item $\psi_2(n) = n^2 + 2n - 1$ is prime $\iff \nexists \begin{pmatrix} a & b \\
    c & d \\
    \end{pmatrix} \in SL_2(\mathbf{N})$ such that $\max \{ac, bd\} +2bc - \min \{ac, bd\} = n$ \\
    \item $\phi_3(n) = n^2 + 3n + 1$ is prime $\iff \nexists \begin{pmatrix} a & b \\
    c & d \\
    \end{pmatrix} \in SL_2(\mathbf{N})$ such that $ac + 3bc + bd = n$ \\

\end{corollary}

\

For a given $(m, n) \in \mathcal{D}_f$, we describe the procedure for computing $\hat{F}_f^{-1}(m, n) \in SL_2(\mathbf{N}_0)$.

\

\begin{algorithm}\label{algorithm}\label{computinginverses}
Let $f \in \mathbf{Z}[x]$ be enumerable and let $(m,n) \in \mathcal{D}_f$. To find $\hat{F}_f^{-1}(m, n) \in SL_2(\mathbf{N}_0)$ apply the following steps

\begin{enumerate}
    \item $\text{While } (m, n) \neq (1,0) \text{ do:  } (m,n) \rightarrow \bar{c}_f\bar{S}^{- \lfloor \frac{n}{m} \rfloor}(m, n)$.
    \item Record the steps $- \lfloor \frac{n}{m} \rfloor$ to find the unique sequence of $\bar{S}, \bar{c}_f$ that generates $(m,n)$.
    \item Convert this into a sequence of $S$ and $T_f$ using the relation $\bar{T}_f = \bar{c}_f \circ \bar{S} \circ \bar{c}_f$. If necessary, change $(1,0)$ to $\bar{c}_f(1,0)$.
    \item Convert this into a sequence of $S$ and $T$.
    \item Multiply the matrices.
\end{enumerate}

\end{algorithm}

\begin{example}
\end{example}

As a demonstration of Algorithm~\ref{computinginverses}, we do an inverse calculation.

\

We take one of the enumerable polynomials, say  $\phi_1(n) = n^2 + n + 1$. Then $\phi_1(100) = 10101$. One can check that $37 \ \lvert \ 10101$ and so, $(37, 100) \in \mathcal{D}_{\phi_1}$. We know from the uniqueness of invertible, equivariant maps, that $\hat{F}_{\phi_1} = \Phi_1$. Let us compute $\Phi_1^{-1}(37,100)$. 

\

The first step is to reduce $(37, 100)$ to the root pair $(1,0)$ via the operations $\bar{S}^{-1}$ and $\bar{c}_f$. We record the intermediate steps.

\

\

$\bar{S}^{-2}(37, 100) = (37, 100 - 2 \times 37) = (37, 26)$

\

$\bar{c}_f(37,26) = (\frac{26^2 + 26 + 1}{37}, 26) = (19, 26)$.

\

$\bar{S}^{-1}(19, 26) = (19, 7)$.

\

$\bar{c}_f(19,7) = (3,7)$

\

$\bar{S}^{-2}(3,7) = (3,1)$

\

$\bar{c}_f(3,1) = (1,1)$

\

$\bar{S}^{-1}(1,1)  = (1,0)$

\

and we are done. Note that each pair we went through along the way belongs to the set $\mathcal{D}_{\phi_1}$ (by Proposition~\ref{invariance}). Now, since we found that

$$(1,0) = \bar{S}^{-1}\bar{c}_f\bar{S}^{-2}\bar{c}_f\bar{S}^{-1}\bar{c}_f\bar{S}^{-2}(37,100)$$

we apply $\bar{S}$ and $\bar{c}_f$ on the right i.e. undo the inverses to get

$$(37,100) = \bar{S}^2\bar{c}_f\bar{S}\bar{c}_f\bar{S}^2\bar{c}_f\bar{S}(1,0)=\bar{S}^2\bar{c}_f\bar{S}\bar{c}_f\bar{S}^2\bar{c}_f\bar{S}\bar{c}_f(1,0)$$

Note we added an extra $\bar{c}_f$ on the far-right, which is allowed since $\bar{c}_f(1,0) = (1,0)$.

$$(37, 100) = \bar{S}^2\bar{T}_f\bar{S}^2\bar{T}_f(1,0)$$

We can now find the corresponding matrix in $SL_2(\mathbf{N}_0)$

$$\Phi_1^{-1}(37, 100)= \Phi_1^{-1}(\bar{S}^2\bar{T}_f\bar{S}^2\bar{T}_f(1,0)) = S^2TS^2T (I)$$

\

Now compute the matrix product.

$$S^2TS^2T = \begin{pmatrix}
    3 & 4 \\
    8 & 11 \\
\end{pmatrix} \in SL_2(\mathbf{N}_0)$$

Indeed, $(a^2 + ab + b^2, ac + bc + bd) = (37, 100)$.

\

We can also find the $\Phi$ and $\Psi$ family "relatives" of the pair $(37,100) \in \mathcal{D}_{\phi_1}$ which all correspond to the same matrix in $SL_2(\mathbf{N}_0)$.

$$ \Phi_0 \begin{pmatrix}
    3 & 4 \\
    8 & 11 \\
\end{pmatrix} = (25, 68)$$

$$ \Psi_2 \begin{pmatrix}
    3 & 4 \\
    8 & 11 \\
\end{pmatrix} = (31, 84)$$

$$\Phi_3 \begin{pmatrix}
    3 & 4 \\
    8 & 11 \\
\end{pmatrix} = (61, 164)$$

Indeed, as a sanity check we find $\phi_0(68) = 25 \times 185$, $\psi_2(84) = 31 \times 233$, and $\phi_3(164) = 61 \times 449$. Note that in each of these factorizations, the second ("complementary") factor corresponds to the first component when the respective function is applied to the complement matrix

$$c \begin{pmatrix}
    3 & 4 \\
    8 & 11 \\
\end{pmatrix} = \begin{pmatrix}
    11 & 8 \\
    4 & 3 \\
\end{pmatrix}$$

\

Deriving pairs from the same matrix defines a one-to-one correspondence between each of the sets $\mathcal{D}_{\phi_0}$, $\mathcal{D}_{\phi_1}$, $\mathcal{D}_{\psi_2}$, and $\mathcal{D}_{\phi_3}$. At the risk of sounding overly poetic, one could say that $SL_2(\mathbf{N}_0)$ serves as a "highway" for divisibility facts about enumerable polynomials.

\

We give one more application. Denote by $\mathbf{P}$ the set of positive integer primes.

\

\begin{definition}
Let $\mathcal{P}(f) = \{ p \in \mathbf{P} : p \lvert f(n) \text{ for some } n \in \mathbf{N}_0 \}$.
\end{definition}

It is well-known and easy to check that $\mathcal{P}(\phi_0(n)) = \{p \in \mathbf{P}: p \equiv 1 \text{ mod } 4 \} \cup \{2 \}$.

\

Consider $\bar{S}^{\alpha_k} \bar{c}_f \cdots \bar{c}_f \bar{S}^{\alpha_1}\bar{c}_f\bar{S}^{\alpha_0}(1,0)$. What can we say about the sequence of first pair components $(m_i)_{i=0}^k$ induced by the sequence $(\alpha_i)_{i=0}^k$?

\

We consider the sequence defined by $(m_{k}, n_{k}) = \bar{c}_f\bar{S}^{\alpha_{k}}(m_{k-1}, n_{k-1})$ with $(m_0, n_0) := (1, 0)$. Then we have

$$n_{k} = n_{k-1} + \alpha_{k} m_{k-1}$$

$$m_{k} = \frac{|f(n_{k})|}{m_{k-1}}$$

\begin{proposition}\label{recursion}
$$m_{k} = \prod_{i=0}^k |f(n_{k-i})|^{(-1)^{i}}$$
\end{proposition}

\begin{proof}
We know that $m_{k} = \frac{|f(n_k)|}{m_{k-1}}$. Then, recursively, $m_{k+1} = \frac{|f(n_k)|}{\frac{|f(n_{k-1})|}{m_{k-2}}} = \frac{|f(n_k)| m_{k-2}}{|f(n_{k-1})|}$. Continuing in this manner, the proposition follows. 
\end{proof}

\

\begin{proposition}
Let $f \in \mathbf{Z}[x]$ be enumerable. Then for all $p \in \mathcal{P}(f)$ There exist natural numbers $n_0 < n_1 < \cdots < n_k < p$ such that

$$p = \prod_{i=0}^k |f(n_{k-i})|^{(-1)^{i}}$$

In particular, for all $p \equiv 1\text{ mod } 4$, there exist natural numbers $n_0 < n_1 < \cdots < n_k < p$ such that

$$p = \prod_{i=0}^k (n_{k-i}^2 + 1)^{(-1)^{i}}$$
\end{proposition}

\begin{proof}
Recall that $\hat{F}_f: SL_2(\mathbf{N}_0) \rightarrow \mathcal{D}_f$ is an invertible, equivariant map. Then, $p \in \mathcal{P}(f)$ and so $p \lvert f(n)$ for some $n \in \mathbf{N}_0$ meaning $(p, n) \in \mathcal{D}_f$. We know $(p, n) = F(A)$ for some $A \in SL_2(\mathbf{N}_0)$. Then, $F(A) = \bar{S}^{\alpha_k} \bar{c}_f \cdots \bar{c}_f \bar{S}^{\alpha_1}\bar{c}_f\bar{S}^{\alpha_0}(1,0) = (p, n)$. The existence of such naturals then follows from Proposition~\ref{recursion}. The upper bound of $p$ will become obvious when we consider Example~\ref{fraction}.
\end{proof}

\begin{example}\label{fraction}
\end{example}

We compute two such representations of a prime $p \equiv 1 \text{ mod } 4$. We take $p := 113$.

\

Consider the congruence $n^2 + 1 = 0 \text{ mod } 113$

\

This congruence has two solutions, $n = 15$, $98 \text { mod } 113$. Then, we know $(113,15) \in \mathcal{D}_{\phi_0}$ and $(113, 98) \in \mathcal{D}_{\phi_0}$.

Apply the first part of Algorithm~\ref{algorithm} to $(113, 15)$, saving the intermediate pairs

$$(113, 15)$$
$$(2, 15)$$
$$(2, 1)$$
$$(1,1)$$
$$(1,0)$$

Then varying the placement of the $n$ components between the numerator and denominator in the above steps we find

$$113 = \frac{(15^2 + 1)}{(1^2 + 1)}$$

Now do the same for $(113, 98)$.

$$(113, 98)$$
$$(85, 98)$$
$$(85, 13)$$
$$(2, 13)$$
$$(2, 1)$$
$$(1,1)$$
$$(1,0)$$

Varying the $n$ components in the above steps we find

$$113 = \frac{(98^2 + 1)(1^2 + 1)}{(13^2 + 1)}$$

\section{Recursions for the $\Phi_0$ tree and its $\mathcal{S}$-sequence}\label{section_phi0recursions}

We give the first 4 rows of the $\Phi_0$ tree enumeration of $SL_2(\mathbf{N}_0)$, which is obtained either by starting with $(1, 0)$ and applying $\bar{S}$ on the right and $\bar{T}_{\phi_0}$ on the left, or by directly applying the function $\Phi_0: SL_2(\mathbf{N}_0) \rightarrow \mathcal{D}_{\phi_0}$ to the matrix tree of $SL_2(\mathbf{N}_0)$ generated by the matrices $S$ and $T$.

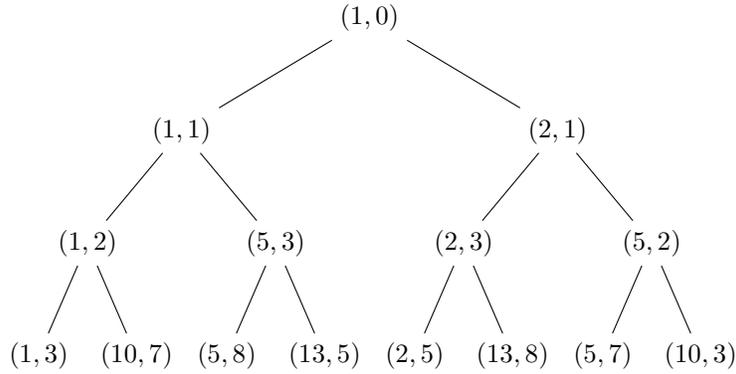
$$
\begin{tikzpicture}[level distance=15mm, 
                    level 1/.style={sibling distance=50mm},
                    level 2/.style={sibling distance=25mm},
                    level 3/.style={sibling distance=13mm}]
  \node {\((1, 0)\)}
    child {
      node {\((1, 1)\)} 
      child {
        node {\((1, 2)\)}
        child {node {\((1, 3)\)}}
        child {node {\((10, 7)\)}}
      }
      child {
        node {\((5, 3)\)}
        child {node {\((5, 8)\)}}
        child {node {\((13, 5)\)}}
      }
    }
    child {
      node {\((2, 1)\)}
      child {
        node {\((2, 3)\)}
        child {node {\((2, 5)\)}}
        child {node {\((13, 8)\)}}
      }
      child {
        node {\((5, 2)\)}
        child {node {\((5, 7)\)}}
        child {node {\((10, 3)\)}}
      } 
    };
\end{tikzpicture}
$$

\begin{remark}\label{landauremark}
Proving that there are infinitely many primes of the form $p = n^2 + 1$ (Conjecture~\ref{landau}), equates to showing that infinitely many $n \in \mathbf{N}$ never appear as the second component of a pair on the interior of the $\Phi_0$ tree. This equivalence follows immediately from Theorem~\ref{main} and the surrounding discussion.
\end{remark}

We study the properties of the $\Phi_0$ tree. Many of the properties we consider extend easily to $\Phi_1, \Psi_2,$ and $\Phi_3$ and can be proved using similar methods. Analogous properties for those trees are given in Section~\ref{section_summary}.

\

We prove several recursions for the row sums and means of $\Phi_0$. Note that we begin indexing rows at $k = 0$.

\

\begin{definition}\label{sumsandavgsdefns}
We define, 

$$M_k(\Phi_0) := \sum_{(m, n) \in \text{row}_{_k}(\Phi_0)} m$$

namely the row sum of the first components

$$N_k(\Phi_0) := \sum_{(m, n) \in \text{row}_{_k}(\Phi_0)} n$$

namely the row sum of the second components

$$R_k(\Phi_0) := \sum_{(m, n) \in \text{row}_{_k}(\Phi_0)} \frac{n}{m}$$

namely the ratio sum of the two components.
\end{definition}

In the following theorem we simply write $M_k$, $N_k,$ and $R_k$ instead of $M_k(\Phi_0)$, $N_k(\Phi_0)$ and $R_k(\Phi_0)$. 

\begin{theorem}\label{phi0recursions}
The following recursions are satisfied

\begin{itemize}
    \item $M_k = 5M_{k-1} - 2M_{k-2} $\ \text{ with initial conditions }$M_0 = 1, \ M_1 = 3$.
    \item $N_k = 5N_{k-1} -  2N_{k-2}$ \text{ with initial conditions }$N_0 = 0, \ N_1 = 2$.
    \item $R_k = R_{k-1} + 3(2^{k-2})$ \text{ with initial condition }$R_0 = 0$.
\end{itemize}
\end{theorem}

\begin{proof}

We will prove that $R_k = R_{k-1} + 3(2^{k-2})$ \text{ with initial condition }$R_0 = 0$. The other two recursions are proved in a similar fashion. The initial condition is simply the ratio sum over row $k=0$. Now, let $r(m, n) := \frac{n}{m}$. Then we can write

\begin{align*}
    R_k   &= \sum_{(m, n) \in \text{row}_{_k}(\Phi_0)} \frac{n}{m} \\
    &= \sum_{\substack{
    i = 1 \\ {(m, n) \in \text{row}_{_k}(\Phi_0)} \\}}^{2^k} r(m_i, n_i) \\
    &= \sum_{\substack{
    i = 1 \\ {(m, n) \in \text{row}_{_k}(\Phi_0)} \\}}^{2^{k-1}} r(m_i, n_i) + r\left(\bar{c}_f(m_i, n_i)\right) \\
    &= \sum_{\substack{
    i = 1 \\ {(m, n) \in \text{row}_{_{k-1}}(\Phi_0)} \\}}^{2^{k-2}} r \left(\bar{S}(m_i, n_i)\right) + r\left(\bar{T}_f(m, n)\right) + r \left(\bar{S}(\bar{c}_f(m_i, n_i))\right) + r\left(\bar{T}_f(\bar{c}_f(m_i, n_i))\right) \\
    &= \sum_{\substack{
    i = 1 \\ {(m, n) \in \text{row}_{_{k-1}}(\Phi_0)} \\}}^{2^{k-2}} \frac{m_i + n_i}{m_i} + \frac{n_i + \frac{n_i^2 + 1}{m_i}}{m_i + 2n_i + \frac{n_i^2 + 1}{m_i}} + \frac{n_i + \frac{n_i^2 + 1}{m_i}}{\frac{n_i^2 + 1}{m_i}} + \frac{n_i+m_i}{\frac{n_i^2 + 1}{m_i} + 2n_i + m_i} \\
    &= \sum_{\substack{
    i = 1 \\ {(m, n) \in \text{row}_{_{k-1}}(\Phi_0)} \\}}^{2^{k-2}} \frac{m_i + n_i}{m_i}  + \frac{n_i + \frac{n_i^2 + 1}{m_i}}{\frac{n_i^2 + 1}{m_i}} + 1 \\
    &= \sum_{\substack{
    i = 1 \\ {(m, n) \in \text{row}_{_{k-1}}(\Phi_0)} \\}}^{2^{k-2}} 1 + \frac{n_i}{m_i}  + \frac{n_i}{\frac{n_i^2 + 1}{m_i}} + 1 + 1 \\
    &= \sum_{\substack{
    i = 1 \\ {(m, n) \in \text{row}_{_{k-1}}(\Phi_0)} \\}}^{2^{k-2}} 3 + r(m_i,n_i) + r(\bar{c}_f(m_i, n_i)) \\
    &= 3(2^{k-2}) + \sum_{\substack{
    i = 1 \\ {(m, n) \in \text{row}_{_{k-1}}(\Phi_0)} \\}}^{2^{k-1}} r(m_i,n_i) \\
    &= 3(2^{k-2}) + R_{k-1}\
\end{align*}
\end{proof}

Using matrix diagonalization on the first two recursions and induction on the third, we find the following closed forms

\

\begin{corollary}\label{closed_forms}

\[
M_k(\Phi_0) = 
\frac{1}{34} \left(-\left(\frac{1}{2} \left(5 - \sqrt{17}\right)\right)^k \left(-17 + \sqrt{17}\right) + \left(\frac{1}{2} \left(5 + \sqrt{17}\right)\right)^k \left(17 + \sqrt{17}\right)\right)
\]

\[
N_k(\Phi_0) = 
\frac{2^{1 - k} \left(-\left(5 - \sqrt{17}\right)^k + \left(5 + \sqrt{17}\right)^k\right)}{\sqrt{17}}
\]

\[
R_k(\Phi_0) = \frac{3}{2} \left (2^k - 1 \right)
\]

\end{corollary}

\

\begin{corollary}\label{phi0average}
The limit of the ratio average over the rows of $\Phi_0$

$$\lim_{k \rightarrow \infty} \frac{R_k(\Phi_0)}{2^k} = \frac{3}{2}$$
\end{corollary}

Let us now define the Calkin-Wilf tree.

\begin{definition}
The Calkin-Wilf Tree, as introduced in \cite{CW}, is generated by the pair maps $\bar{L}:(a, b) \rightarrow (a, a+b)$ (left child map) and $\bar{R}: (a, b) \rightarrow (a+b, b)$ (right child map), beginning with the pair $(1,1)$

$$
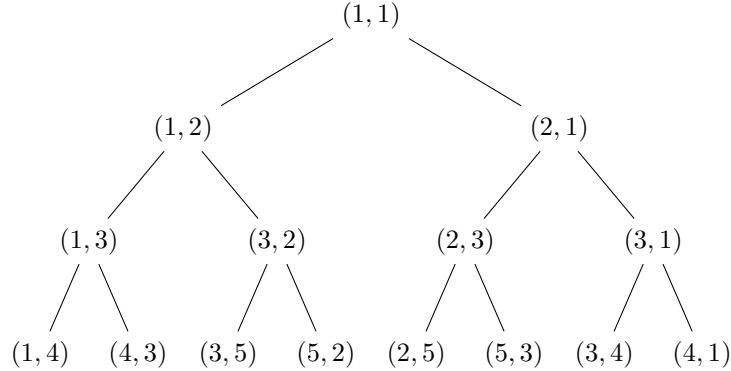
\begin{tikzpicture}[level distance=15mm, 
                    level 1/.style={sibling distance=50mm},
                    level 2/.style={sibling distance=25mm},
                    level 3/.style={sibling distance=13mm}]
  \node {\((1, 1)\)}
    child {
      node {\((1, 2)\)} 
      child {
        node {\((1, 3)\)}
        child {node {\((1, 4)\)}}
        child {node {\((4, 3)\)}}
      }
      child {
        node {\((3, 2)\)}
        child {node {\((3, 5)\)}}
        child {node {\((5, 2)\)}}
      }
    }
    child {
      node {\((2, 1)\)}
      child {
        node {\((2, 3)\)}
        child {node {\((2, 5)\)}}
        child {node {\((5, 3)\)}}
      }
      child {
        node {\((3, 1)\)}
        child {node {\((3, 4)\)}}
        child {node {\((4, 1)\)}}
      } 
    };
\end{tikzpicture}$$

We denote the Calkin-Wilf Tree as $\mathcal{CW}$.
\end{definition}

\begin{remark}
The Calkin-Wilf tree has the same limiting ratio average over rows as $\Phi_0$ (\cite{CW_means}, Theorem 1)

$$\lim_{k \rightarrow \infty} \frac{R_k(\Phi_0)}{2^k} = \lim_{k \rightarrow \infty} \frac{R_k(\mathcal{CW})}{2^k} = \frac{3}{2}$$
\end{remark}

\begin{remark}\label{calkinwilfrationals}
If we view the pairs $(a, b)$ on $\mathcal{CW}$ as fractions $\frac{a}{b}$, then the Calkin-Wilf tree read-off left to right, row by row, enumerates the positive rationals $\mathbf{Q}_{>0}$, meaning each $q \in \mathbf{Q}_{>0}$ appears exactly once in the sequence

$$\left\{ \frac{1}{1}, \frac{1}{2}, \frac{2}{1}, \frac{1}{3}, \frac{3}{2}, \frac{2}{3}, \cdots \right\}$$

A simple proof of this fact can be found in \cite{CW}.
\end{remark}

Calkin and Wilf also proved that the integer sequence of first components $\{ a(n) \}_{n \in \mathbf{N}}$ of $\mathcal{CW}$ read off row-by-row is given by the recursion 

$$\begin{cases}
a(2n) = a(n) \\
a(2n+1) = a(n) + a(n+1) \\
\end{cases}$$

with initial conditions $a(1) = 0$ and $a(2) = 1$, and with the initial first component on $\mathcal{CW}$ taken to be $a(2)$. The sequence $\{ a(n) \}_{n \in \mathbf{N}}$ is known as Stern's diatomic sequence (OEIS A002487) and has a number of remarkable properties, some of which are given in \cite{CW}. Moreover, the pairs on $\mathcal{CW}$ can be generated row-by-row by running over the sequence $a$ as

$$\left(a(n), a(n+1)\right)$$

Our goal is to find an integer sequence that  analogously generates the divisor pair tree $\Phi_0$.

\begin{definition}
Let $\mathcal{S}(k)$ denote the integer sequence which results from reading off the second components of $\Phi_0$. From the first $4$ rows of the second components

$$
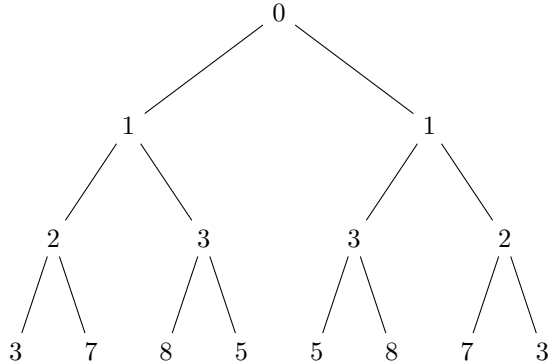
\begin{tikzpicture}[level distance=1.5cm,
  level 1/.style={sibling distance=4cm},
  level 2/.style={sibling distance=2cm},
  level 3/.style={sibling distance=1cm}]
  \node {$0$}
    child {node {$1$}
      child {node {$2$}
        child {node {$3$}}
        child {node {$7$}}
      }
      child {node {$3$}
        child {node {$8$}}
        child {node {$5$}}
      }
    }
    child {node {$1$}
      child {node {$3$}
        child {node {$5$}}
        child {node {$8$}}
      }
      child {node {$2$}
        child {node {$7$}}
        child {node {$3$}}
      }
    };
\end{tikzpicture}
$$

we obtain

$$\{ \mathcal{S}(k) \}_{k \in \mathbf{N}} = \{0,1,1,2,3,3,2,3,7,8,5,5,8,7,3, \cdots \}$$
\end{definition}

\begin{remark}
Note that the boundary of the second component tree corresponds to the $\mathcal{S}$-sequence values $\mathcal{S}(2^n)$ on the left and  $\mathcal{S}(2^{n+1} - 1)$ on the right, for $n \in \mathbf{N}$.
\end{remark}

\begin{proposition}\label{n_components_net}
The following "net" recursively generates the second components of $\Phi_0$

$$
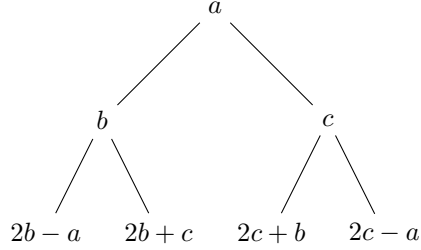
\begin{tikzpicture}[level distance=1.5cm,
  level 1/.style={sibling distance=3cm},
  level 2/.style={sibling distance=1.5cm}]
  \node {$a$}
    child {node {$b$}
      child {node {$2b-a$}}
      child {node {$2b+c$}}
    }
    child {node {$c$}
      child {node {$2c+b$}}
      child {node {$2c-a$}}
    };
\end{tikzpicture}$$
The generation process is as follows:
we begin with initial conditions $(a, b, c) = (0, 1, 1)$. On each row $k$ we apply the net to each node and its two children on row $k+1$ to generate the row $k+2$. 
\end{proposition}

\begin{proof}
We denote by $n_S, \ n_T, \ n_{S^2}, \ n_{TS}, \ n_{ST}, \ n_{T^2}$ the second components of the image pairs $\bar{S}(m, n), \ \bar{T}_{\phi_0}(m, n), \ \bar{S}^2(m, n), \ \bar{T}_{\phi_0}\left(\bar{S}(m, n)\right), \ \bar{S}\left(\bar{T}_{\phi_0}(m, n)\right), \ \bar{T}^2_{\phi_0}(m, n)$ respectively. Then, beginning with an arbitrary second component $n$, we get that its children and grandchildren will be

$$
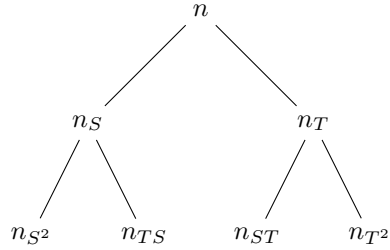
\begin{tikzpicture}[level distance=1.5cm,
  level 1/.style={sibling distance=3cm},
  level 2/.style={sibling distance=1.5cm}]
  \node {$n$}
    child {node {$n_S$}
      child {node {$n_{S^2}$}}
      child {node {$n_{TS}$}}
    }
    child {node {$n_T$}
      child {node {$n_{ST}$}}
      child {node {$n_{T^2}$}}
    };
\end{tikzpicture}$$

After computing all of the image pairs, $$\bar{S}(m, n), \ \bar{T}_{\phi_0}(m, n), \ \bar{S}^2(m, n), \ \bar{T}_{\phi_0}\left(\bar{S}(m, n)\right), \ \bar{S}\left(\bar{T}_{\phi_0}(m, n)\right), \ \bar{T}^2_{\phi_0}(m, n)$$ and writing out their second components, one can verify that the last row of the net can be restated in terms of linear combinations of values on the previous two rows as

$$
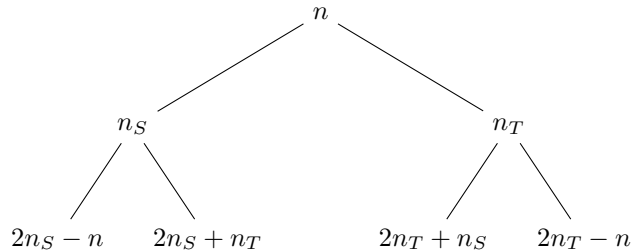
\begin{tikzpicture}[level distance=1.5cm,
  level 1/.style={sibling distance=5cm},
  level 2/.style={sibling distance=2cm}]
  \node {$n$}
    child {node {$n_S$}
      child {node {$2n_S -  n$}}
      child {node {$2n_S + n_T$}}
    }
    child {node {$n_T$}
      child {node {$2n_T + n_S$}}
      child {node {$2n_T - n$}}
    };
\end{tikzpicture}$$

Rewriting $n, \ n_S, \ n_T$ as $a, \ b, \ c$ respectively, we get the claim of the proposition.
\end{proof}

To reiterate, the $\Phi_0$ tree of divisor pairs has the properties

\begin{itemize}
    \item Every integer $n \in \mathbf{N}_0$ appears on the tree exactly $\tau(n^2 + 1)$ times. 
    \item An integer $n \in \mathbf{N}_0$ is absent from the interior of the tree if and only if $n^2 + 1$ is prime.
\end{itemize}

By using the recursion in Proposition~\ref{n_components_net}, we can generate $\Phi_0$ using two linear maps.

\

We want a matrix $L$ (left-child matrix) that sends

\

$\begin{pmatrix}
a \\
b \\
c \\
\end{pmatrix} \rightarrow \begin{pmatrix}
b \\
2b-a \\
2b+c \\
\end{pmatrix}$

\

and a matrix $R$ (right-child matrix) that sends

\

$\begin{pmatrix}
a \\
b \\
c \\
\end{pmatrix} \rightarrow \begin{pmatrix}
c \\
2c+b \\
2c-a \\
\end{pmatrix}$

\

These conditions are satisfied by the matrices

$$L = 
\begin{pmatrix}
0 & 1 & 0 \\
-1 & 2 & 0 \\
0 & 2 & 1 \\
\end{pmatrix}
$$

$$R = \begin{pmatrix}
0 & 0 & 1 \\
0 & 1 & 2 \\
-1 & 0 & 2 \\
\end{pmatrix}$$

\begin{remark}
$L, R \in SL_3(\mathbf{Z})$ and are  mutual conjugates via the matrix

$$\begin{pmatrix}
1 & 0 & 0 \\
0 & 0 & 1 \\
0 & 1 & 0 \\
\end{pmatrix}$$

which has determinant $-1$.
\end{remark}

Beginning with the initial condition vector $
= \begin{pmatrix}
0 \\
1 \\
1 \\
\end{pmatrix}$ we compose by $L$ on the left and by $R$ on the right. This yields the vector tree for $\Phi_0$ (first $3$ rows)

\

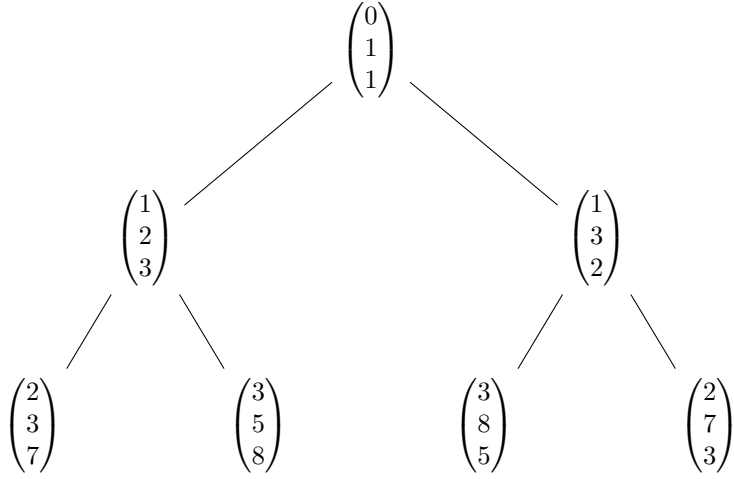
$$\begin{tikzpicture}[level distance=2.5cm,
  level 1/.style={sibling distance=6cm},
  level 2/.style={sibling distance=3cm}]
  \node {$\begin{pmatrix} 0 \\ 1 \\ 1 \end{pmatrix}$}
    child {node {$\begin{pmatrix} 1 \\ 2 \\ 3 \end{pmatrix}$}
      child {node {$\begin{pmatrix} 2 \\ 3 \\ 7 \end{pmatrix}$}}
      child {node {$\begin{pmatrix} 3 \\ 5 \\ 8 \end{pmatrix}$}}
    }
    child {node {$\begin{pmatrix} 1 \\ 3 \\ 2 \end{pmatrix}$}
      child {node {$\begin{pmatrix} 3 \\ 8 \\ 5 \end{pmatrix}$}}
      child {node {$\begin{pmatrix} 2 \\ 7 \\ 3 \end{pmatrix}$}}
    };
\end{tikzpicture}$$

\

One sees that the vectors on this tree all have the form

$$\begin{pmatrix}
n \\
n_S \\ 
n_T \\
\end{pmatrix} = \begin{pmatrix}
n \\
n+m \\ 
n + \frac{n^2+1}{m}  \\
\end{pmatrix}$$

Therefore, one easily recovers the tree $\Phi_0$ by sending each vector

$$\begin{pmatrix}
a \\
b \\ 
c \\
\end{pmatrix} \rightarrow 
(b-a, a) \in \mathcal{D}_{\phi_0}$$

Note that the top component of each vector runs over $\mathcal{S}(k)$ when reading off the sequence of vectors row-by-row. In fact, each vector will have the form

$$\mathbf{v} = \begin{pmatrix}
\mathcal{S}(k) \\
\mathcal{S}(2k) \\ 
\mathcal{S}(2k+1) \\
\end{pmatrix}$$

\

while its left and right children will have the forms

$$\begin{pmatrix}
\mathcal{S}(2k) \\
\mathcal{S}(4k) \\ 
\mathcal{S}(4k+1) \\
\end{pmatrix} \text{ and } \begin{pmatrix}
\mathcal{S}(2k+1) \\
\mathcal{S}(4k+2) \\ 
\mathcal{S}(4k+3) \\
\end{pmatrix}$$

respectively. Solving the equations

$$L\mathbf{v} = \begin{pmatrix}
\mathcal{S}(2k) \\
\mathcal{S}(4k) \\ 
\mathcal{S}(4k+1) \\
\end{pmatrix} \text{ and } R\mathbf{v} = \begin{pmatrix}
\mathcal{S}(2k+1) \\
\mathcal{S}(4k+2) \\ 
\mathcal{S}(4k+3)
\end{pmatrix}$$

yields the set of recursions for $\mathcal{S}(k)$

$$
\begin{cases}
\mathcal{S}(4k) = 2\mathcal{S}(2k) - \mathcal{S}(k) \\
\mathcal{S}(4k+1) = 2\mathcal{S}(2k) + \mathcal{S}(2k+1) \\
\mathcal{S}(4k+2) =  2\mathcal{S}(2k+1) + \mathcal{S}(2k) \\
\mathcal{S}(4k+3) =  2\mathcal{S}(2k+1) - \mathcal{S}(k)  \\
\end{cases}
$$

\

with initial conditions $\mathcal{S}(1) = 0$, $\mathcal{S}(2) = 1$, $\mathcal{S}(3) = 1$. 

\

One may also ask about the sequence that results from reading off the first components  from the divisor pair tree $\Phi_0$. Since the difference $b-a$ of the second and first vector components gave us the first divisor pair component $m$, and the vectors each had the form

$$\mathbf{v} = \begin{pmatrix}
\mathcal{S}(k) \\
\mathcal{S}(2k) \\ 
\mathcal{S}(2k+1) \\
\end{pmatrix}$$

we find that the first components are given by the relation $\mathcal{S}(2k) - \mathcal{S}(k)$ for all $k \in \mathbf{N}$. Thus, reading off the divisor pair tree $\Phi_0$ row-by-row the $k$th pair will be $\left( \mathcal{S}(2k) - \mathcal{S}(k), \mathcal{S}(k) \right)$ for $k \in \mathbf{N}$.

\

We define $k$-regular sequences, as introduced by Allouche and Shallit \cite{q_regular_sequences}. There are many equivalent definitions of $k$-regular sequences. The one we state allows us to immediately conclude that $\mathcal{S}$ is a $2$-regular sequence

\begin{definition}
A sequence \(s(n)\) is \(k\)-regular if there exists an integer \(E\) such that, for all \(e_j > E\) and \(0 \leq r_j \leq k^{e_j} - 1\), every subsequence of \(s\) of the form \(s(k^{e_j}n + r_j)\) is expressible as a linear combination
\[
\sum_{i}c_{ij}s(k^{f_{ij}}n+b_{ij}),
\]
where \(c_{ij}\) is an integer, \(f_{ij} \leq E\), and \(0 \leq b_{ij} \leq k^{f_{ij}} - 1.\)
\end{definition}

The asymptotic analysis of $k$-regular sequences is an area that has undergone significant development in recent years \cite{Asymptotics}, \cite{Asymptotics2}, \cite{Asymptotics3}.

\section{Summary of enumerable polynomials}\label{section_summary}

The purpose of this section is to record the most important properties of each enumerable polynomial in a single place. Recall that each polynomial's divisor pair tree was initially obtained by applying the equivariant, invertible map $\hat{F}_f$ to the $SL_2(\mathbf{N}_0)$ matrix tree generated by $S$ and $T$

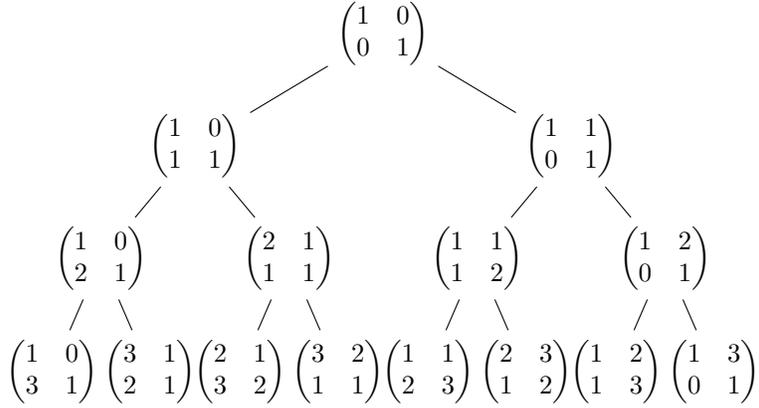
$$\begin{tikzpicture}[level distance=15mm, 
                    level 1/.style={sibling distance=50mm},
                    level 2/.style={sibling distance=25mm},
                    level 3/.style={sibling distance=13mm}]
  \node {\(\begin{pmatrix}1 & 0 \\ 0 & 1\end{pmatrix}\)}
    child {
      node {\(\begin{pmatrix}1 & 0 \\ 1 & 1\end{pmatrix}\)} 
      child {
        node {\(\begin{pmatrix}1 & 0 \\ 2 & 1\end{pmatrix}\)}
        child {node {\(\begin{pmatrix}1 & 0 \\ 3 & 1\end{pmatrix}\)}}
        child {node {\(\begin{pmatrix}3 & 1 \\ 2 & 1\end{pmatrix}\)}}
      }
      child {
        node {\(\begin{pmatrix}2 & 1 \\ 1 & 1\end{pmatrix}\)}
        child {node {\(\begin{pmatrix}2 & 1 \\ 3 & 2\end{pmatrix}\)}}
        child {node {\(\begin{pmatrix}3 & 2 \\ 1 & 1\end{pmatrix}\)}}
      }
    }
    child {
      node {\(\begin{pmatrix}1 & 1 \\ 0 & 1\end{pmatrix}\)}
      child {
        node {\(\begin{pmatrix}1 & 1 \\ 1 & 2\end{pmatrix}\)}
        child {node {\(\begin{pmatrix}1 & 1 \\ 2 & 3\end{pmatrix}\)}}
        child {node {\(\begin{pmatrix}2 & 3 \\ 1 & 2\end{pmatrix}\)}}
      }
      child {
        node {\(\begin{pmatrix}1 & 2 \\ 0 & 1\end{pmatrix}\)}
        child {node {\(\begin{pmatrix}1 & 2 \\ 1 & 3\end{pmatrix}\)}}
        child {node {\(\begin{pmatrix}1 & 3 \\ 0 & 1\end{pmatrix}\)}}
      } 
    };
\end{tikzpicture}$$

 We reuse the notation $\mathcal{S}$ for every sequence generating the corresponding divisor pair tree. It should be clear from the context which polynomial's $\mathcal{S}$-sequence we are referring to. Note that for each $f \in \{\phi_0, \phi_1, \psi_2, \phi_3 \}$ the $\mathcal{S}$-sequences have the properties

\begin{itemize}
    \item $\mathcal{S}$ is $2$-regular.
    \item $f(n)$ is prime if and only if $\mathcal{S}^{-1}(\{n\}) = \{2^n, 2^{n+1} - 1 \}$.
     \item $|\mathcal{S}^{-1}(\{n \})| = \tau\left(f(n)\right)$.
     \item the $k$th pair of the divisor pair tree read off row-by-row is given by $\left( \mathcal{S}(2k) - \mathcal{S}(k), \ \mathcal{S}(k) \right)$ where $k \in \mathbf{N}$.
 \end{itemize}
 
\subsection{$\phi_0(n) = n^2 + 1$}

\textbf{Invertible, equivariant map}

$$\Phi_0 \begin{pmatrix}
    a & b \\
    c & d \\
\end{pmatrix} = (a^2 + b^2, ac + bd)$$

\textbf{Divisor pair tree}

$$
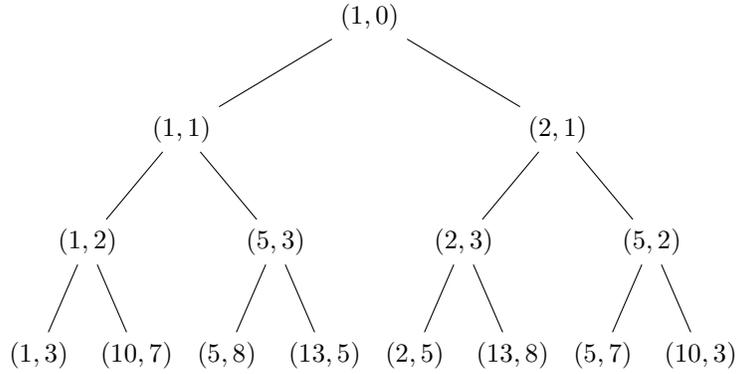
\begin{tikzpicture}[level distance=15mm, 
                    level 1/.style={sibling distance=50mm},
                    level 2/.style={sibling distance=25mm},
                    level 3/.style={sibling distance=13mm}]
  \node {\((1, 0)\)}
    child {
      node {\((1, 1)\)} 
      child {
        node {\((1, 2)\)}
        child {node {\((1, 3)\)}}
        child {node {\((10, 7)\)}}
      }
      child {
        node {\((5, 3)\)}
        child {node {\((5, 8)\)}}
        child {node {\((13, 5)\)}}
      }
    }
    child {
      node {\((2, 1)\)}
      child {
        node {\((2, 3)\)}
        child {node {\((2, 5)\)}}
        child {node {\((13, 8)\)}}
      }
      child {
        node {\((5, 2)\)}
        child {node {\((5, 7)\)}}
        child {node {\((10, 3)\)}}
      } 
    };
\end{tikzpicture}
$$

$\mathcal{S}$-\textbf{sequence}

\

For $k \in \mathbf{N}$,

$$
\begin{cases}
\mathcal{S}(4k) = 2\mathcal{S}(2k) - \mathcal{S}(k) \\
\mathcal{S}(4k+1) = 2\mathcal{S}(2k) + \mathcal{S}(2k+1) \\
\mathcal{S}(4k+2) =  2\mathcal{S}(2k+1) + \mathcal{S}(2k) \\
\mathcal{S}(4k+3) =  2\mathcal{S}(2k+1) - \mathcal{S}(k)  \\
\end{cases}
$$

with initial conditions $\mathcal{S}(1) = 0$, $\mathcal{S}(2) = 1$, $\mathcal{S}(3) = 1$.

\

\subsection{$\phi_1(n) = n^2 + n + 1$}

\textbf{Invertible, equivariant map}

$$\Phi_1 \begin{pmatrix}
    a & b \\
    c & d \\
\end{pmatrix} = (a^2 + ab + b^2, ac + bc + bd)$$

\textbf{Divisor pair tree:}

$$
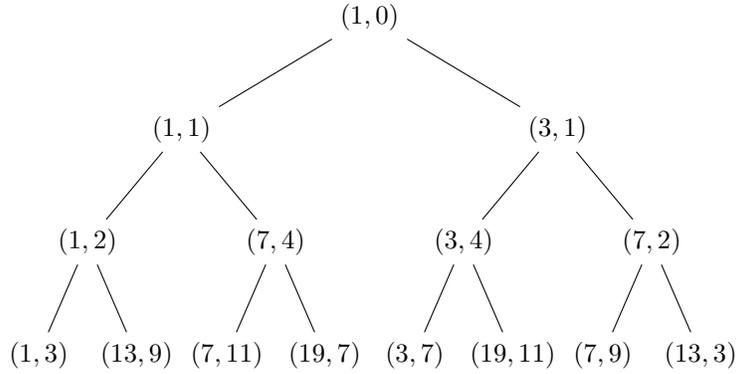
\begin{tikzpicture}[level distance=15mm, 
                    level 1/.style={sibling distance=50mm},
                    level 2/.style={sibling distance=25mm},
                    level 3/.style={sibling distance=13mm}]
  \node {\((1, 0)\)}
    child {
      node {\((1, 1)\)} 
      child {
        node {\((1, 2)\)}
        child {node {\((1, 3)\)}}
        child {node {\((13, 9)\)}}
      }
      child {
        node {\((7, 4)\)}
        child {node {\((7, 11)\)}}
        child {node {\((19, 7)\)}}
      }
    }
    child {
      node {\((3, 1)\)}
      child {
        node {\((3, 4)\)}
        child {node {\((3, 7)\)}}
        child {node {\((19, 11)\)}}
      }
      child {
        node {\((7, 2)\)}
        child {node {\((7, 9)\)}}
        child {node {\((13, 3)\)}}
      } 
    };
\end{tikzpicture}
$$

$\mathcal{S}$-\textbf{sequence}

\

For $k \in \mathbf{N}$,

$$
\begin{cases}
\mathcal{S}(4k) = 2\mathcal{S}(2k) - \mathcal{S}(k) \\
\mathcal{S}(4k+1) = 2\mathcal{S}(2k) + \mathcal{S}(2k+1) + 1 \\
\mathcal{S}(4k+2) =  2\mathcal{S}(2k+1) + \mathcal{S}(2k) + 1 \\
\mathcal{S}(4k+3) =  2\mathcal{S}(2k+1) - \mathcal{S}(k)  \\
\end{cases}
$$
with initial conditions $\mathcal{S}(1) = 0$, $\mathcal{S}(2) = 1$, $\mathcal{S}(3) = 1$.

\

\subsection{$\psi_2(n) = n^2 + 2n - 1$}

\textbf{Invertible, equivariant map}

$$\Psi_2 \begin{pmatrix}
    a & b \\
    c & d \\
\end{pmatrix} = 
(\max \{a, b\}^2 + 2ab - \min \{a, b \}^2,  \ \max \{ac, bd \} + 2bc - \min \{ac, bd\})$$

\textbf{Divisor pair tree:}

$$
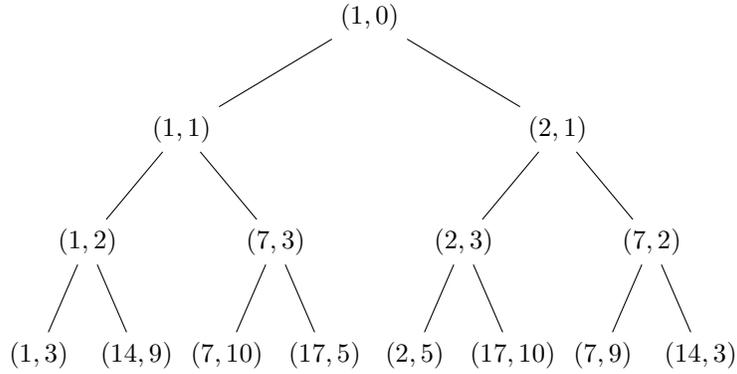
\begin{tikzpicture}[level distance=15mm, 
                    level 1/.style={sibling distance=50mm},
                    level 2/.style={sibling distance=25mm},
                    level 3/.style={sibling distance=13mm}]
  \node {\((1, 0)\)}
    child {
      node {\((1, 1)\)} 
      child {
        node {\((1, 2)\)}
        child {node {\((1, 3)\)}}
        child {node {\((14, 9)\)}}
      }
      child {
        node {\((7, 3)\)}
        child {node {\((7, 10)\)}}
        child {node {\((17, 5)\)}}
      }
    }
    child {
      node {\((2, 1)\)}
      child {
        node {\((2, 3)\)}
        child {node {\((2, 5)\)}}
        child {node {\((17, 10)\)}}
      }
      child {
        node {\((7, 2)\)}
        child {node {\((7, 9)\)}}
        child {node {\((14, 3)\)}}
      } 
    };
\end{tikzpicture}
$$

$\mathcal{S}$-\textbf{sequence}

\

For $k \geq 2$,

$$
\begin{cases}
\mathcal{S}(4k) = 2\mathcal{S}(2k) - \mathcal{S}(k) \\
\mathcal{S}(4k+1) = 2\mathcal{S}(2k) + \mathcal{S}(2k+1) + 2 \\
\mathcal{S}(4k+2) =  2\mathcal{S}(2k+1) + \mathcal{S}(2k) + 2 \\
\mathcal{S}(4k+3) =  2\mathcal{S}(2k+1) - \mathcal{S}(k)  \\
\end{cases}
$$

with initial conditions $\mathcal{S}(1) = 0$, $\mathcal{S}(2) = 1$, $\mathcal{S}(3) = 1$, $\mathcal{S}(4) = 2$, $\mathcal{S}(5) = 3$, $\mathcal{S}(6) = 3$, $\mathcal{S}(7) = 2$.

\begin{remark}
The reason we begin the recursion at $k=2$ is that $\psi_2(0) = -1$ instead of $1$, which initially alters the net generating the second component tree.
\end{remark}

\subsection{$\phi_3(n) = n^2 + 3n + 1$}

\textbf{Invertible, equivariant map}

$$\Phi_3 \begin{pmatrix}
    a & b \\
    c & d \\
\end{pmatrix} = (a^2 + 3ab + b^2, ac + 3bc + bd)$$

\textbf{Divisor pair tree:}

$$
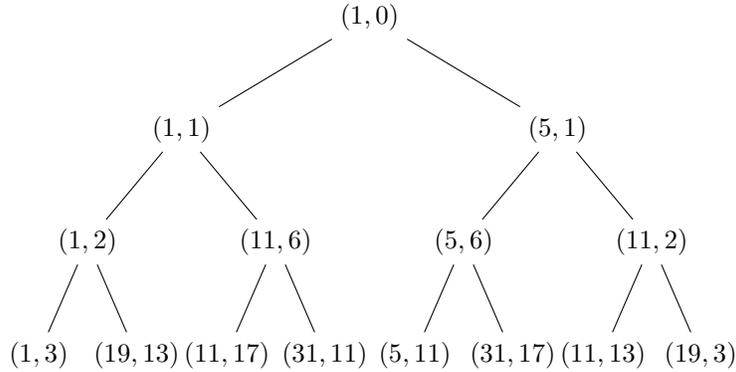
\begin{tikzpicture}[level distance=15mm, 
                    level 1/.style={sibling distance=50mm},
                    level 2/.style={sibling distance=25mm},
                    level 3/.style={sibling distance=13mm}]
  \node {\((1, 0)\)}
    child {
      node {\((1, 1)\)} 
      child {
        node {\((1, 2)\)}
        child {node {\((1, 3)\)}}
        child {node {\((19, 13)\)}}
      }
      child {
        node {\((11, 6)\)}
        child {node {\((11, 17)\)}}
        child {node {\((31, 11)\)}}
      }
    }
    child {
      node {\((5, 1)\)}
      child {
        node {\((5, 6)\)}
        child {node {\((5, 11)\)}}
        child {node {\((31, 17)\)}}
      }
      child {
        node {\((11, 2)\)}
        child {node {\((11, 13)\)}}
        child {node {\((19, 3)\)}}
      } 
    };
\end{tikzpicture}
$$

$\mathcal{S}$-\textbf{sequence}

\

For $k \in \mathbf{N}$,

$$
\begin{cases}
\mathcal{S}(4k) = 2\mathcal{S}(2k) - \mathcal{S}(k) \\
\mathcal{S}(4k+1) = 2\mathcal{S}(2k) + \mathcal{S}(2k+1) + 3 \\
\mathcal{S}(4k+2) =  2\mathcal{S}(2k+1) + \mathcal{S}(2k) + 3 \\
\mathcal{S}(4k+3) =  2\mathcal{S}(2k+1) - \mathcal{S}(k)  \\
\end{cases}
$$

with initial conditions $\mathcal{S}(1) = 0$, $\mathcal{S}(2) = 1$, $\mathcal{S}(3) = 1$.

\

\

\

\textit{Acknowledgements.} I want to thank my advisors Dr. Ingalls and Dr. Logan whose insightful comments and suggestions made this project possible. I also want to thank Christian Kudeba for the many fruitful discussions.

\end{document}